\documentclass{article}

\usepackage[utf8]{inputenc}
\usepackage{lipsum}
\usepackage{url}
\usepackage{mathpazo}
\usepackage[T1]{fontenc}
\usepackage[utf8]{inputenc}
\usepackage{bigints}
\usepackage[english]{babel}
\usepackage[active]{srcltx}
\usepackage{colorspace}
\usepackage{tikz}
\usepackage{sansmath}
\usetikzlibrary{shadings,intersections,patterns}
\usepackage{pgfplots}
\usepackage{soul}
\usepackage{mathtools}

\pgfplotsset{compat=1.10}
\usepgfplotslibrary{fillbetween}
\usetikzlibrary{patterns}
\usetikzlibrary{intersections, calc, angles}
\usepackage{tkz-euclide}
\pgfdeclarelayer{pre main}
\usepackage{environ}
\makeatletter
\newsavebox{\measure@tikzpicture}
\NewEnviron{scaletikzpicturetowidth}[1]{%
  \def\tikz@width{#1}%
  \begin{lrbox}{\measure@tikzpicture}%
  \BODY
  \end{lrbox}%
  \pgfmathparse{#1/\wd\measure@tikzpicture}%
  \BODY
}
\makeatother
\usetikzlibrary{arrows,automata}
\usetikzlibrary{positioning}
\usetikzlibrary{calc,through}
\usepackage{etoolbox}%
\apptocmd{\thebibliography}{\fontsize{11}{15}\selectfont}{}{}%

\tikzset{
    state/.style={
           rectangle,
           rounded corners,
           draw=black, very thick,
           minimum height=2em,
           inner sep=2pt,
           text centered,
           },
}

\usepackage{mathrsfs}

\usepackage[left=2cm, right=2cm, top=3cm, bottom=3cm]{geometry}
\usepackage{amsmath}
\usepackage{amssymb}
\usepackage{amsthm}
\usepackage{hyperref}
\usepackage{graphicx}
\usepackage{subcaption}
\usepackage{caption}
\usepackage{xcolor}
\usepackage{longtable}
\usepackage{verbatimbox}

\usepackage[title]{appendix}
\usepackage{bm}
\usepackage{empheq}
\usepackage{placeins}
\usepackage{anyfontsize}
\theoremstyle{plain}
\newtheorem{theorem}{Theorem}[section]
\newtheorem{lemma}[theorem]{Lemma}
\newtheorem{proposition}[theorem]{Proposition}

\newtheorem{remark}[theorem]{Remark}

\theoremstyle{definition}
\theoremstyle{remark}
\numberwithin{equation}{section}

\newcommand{\E}{\mathcal{E}}

\newcommand{\abs}[1]{\left\lvert#1\right\rvert}
\newcommand{\tens}[1]{\mathsf{#1}}
\newcommand{\n}{\mathbf{n}}

\newcommand{\R}{\mathbb{R}}
\newcommand{\N}{\mathbb{N}}

\newcommand{\vect}[1]{\boldsymbol{#1}}
\theoremstyle{plain}% default
\newtheorem*{theorem*}{Theorem}
\newtheorem*{corollary*}{Corollary}

\theoremstyle{definition}

\newtheorem*{notation*}{Notation}

\numberwithin{figure}{section}

\definecolor{myred}{rgb}{0.9,0,0}

\definecolor{vargreen}{rgb}{0.0, 0.5, 0.0}


\renewcommand{\rho}{\varrho}
\renewcommand{\theta}{\vartheta}

\newcommand{\eps}{\varepsilon}

\usepackage[small, sc]{titlesec}

\usepackage[
backend=biber,
style=numeric,
sorting=nty,
maxbibnames=50,
giveninits=true
]{biblatex}
\usepackage{csquotes}
\renewbibmacro{in:}{%
 \ifentrytype{article}{}{}}
\DeclareFieldFormat[article]{volume}{\textbf{#1}}
\DeclareFieldFormat[article]{title}{\textit{#1}}
\DeclareFieldFormat[article]{journaltitle}{#1}
\DeclareFieldFormat[article]{number}{\textbf{#1}}
\renewbibmacro*{volume+number+eid}{%
  \printfield{volume}%
\setunit*{\textbf{\addcolon}}%<---- was \setunit*{\adddot}%
  \printfield{number}\setunit{\addcomma\space}%
  \printfield{eid}}

\usepackage{xcolor}
\hypersetup{
colorlinks,
linkcolor={red!50!black},
citecolor={blue!50!black},
urlcolor={blue!80!black}
}

\begin{document}
\title{\textsc{A variational analysis of nematic axisymmetric films:\\
the covariant derivative case}}
\author{\textsc{G.\ Bevilacqua}$^1$\thanks{\href{mailto:giulia.bevilacqua@dm.unipi.it}{\texttt{giulia.bevilacqua@dm.unipi.it}}}\,\,\,$-$\,\, \textsc{C.\ Lonati}$^2$\thanks{\href{mailto:chiara.lonati@polito.it}{\texttt{chiara.lonati@polito.it}}}\,\,\,$-$\,\, \textsc{L.\ Lussardi}$^2$\thanks{\href{mailto:luca.lussardi@polito.it}{
\texttt{luca.lussardi@polito.it}}}\,\,\,$-$\,\,\textsc{A.\  Marzocchi}$^3$\thanks{\href{mailto:alfredo.marzocchi@unicatt.it}{
\texttt{alfredo.marzocchi@unicatt.it}}}
\bigskip\\
\normalsize$^1$ Dipartimento di Matematica, Università di Pisa, Largo Bruno Pontecorvo 5, I–56127 Pisa, Italy.\\
\normalsize$^2$  DISMA "Giuseppe Luigi Lagrange", Politecnico di Torino, c.so Duca degli Abruzzi 24, I-10129 Torino, Italy.\\
\normalsize$^3$ Dipartimento di Matematica e Fisica ``N. Tartaglia", Università Cattolica del Sacro Cuore,\\
\normalsize via della Garzetta 48, I-25133 Brescia, Italy\\
}
\date{}

\maketitle

\begin{abstract}
\noindent
Nematic surfaces are thin fluid structures, ideally two-dimensional, endowed with an in-plane nematic order. In 2012, two variational models have been introduced by Giomi \cite{giomi2012hyperbolic} and by Napoli and Vergori \cite{napoli2012surface, napoli2018influence}. Both penalize the area of the surface and the gradient of the director: in \cite{giomi2012hyperbolic} the covariant derivative of the director is considered, while \cite{napoli2018influence} deals with the surface gradient. In this paper, a complete variational analysis of the model proposed by Giomi is performed for revolution surfaces spanning two coaxial rings.
\end{abstract}

\bigskip
\bigskip



\bigskip

\tableofcontents

\section{Introduction}

Motivated by the Oseen-Frank theory for nematic liquid crystals \cite{lin1991static, Virga1994}, in this paper we consider the energy functional
\begin{equation}
\label{eq:funzionale_generale}
    \mathcal{E}(\n, S)=\int_S \left(\gamma+\frac{\kappa}{2}|\nabla \n|^2\right)dA,
\end{equation}
where $S$ is the surface in whose tangent space the unit nematic vector $\n$ is constrained to lie, $\gamma$ is the constant surface tension and $\kappa$ is the nematic constant.
Such a surface $S$ has a prescribed boundary curve and $\nabla \n$ is a gradient of $\n$ as specified later. 
Notice that when $\kappa=0$, the problem \eqref{eq:funzionale_generale} reduces to the classical Plateau problem (see \cite{david2014chapter} and for more recent developments \cite{giusteri2017solution, bevilacqua2019soap, bevilacqua2020dimensional, palmer2021minimal}).

We stress that, in our case, the domain of the nematic vector is not assigned \emph{a priori} as in the classical theory but $\bf n$
is constrained to lie on a free surface. Understanding the optimal shape of the surface and their properties is clearly of paramount interest for possible industrial applications, like flexible screens and lenses. The interesting aspect is the competition between the area term, penalized by constant surface tension, and the nematic energy which prefers flat surfaces, namely the director $\n$ may be constant. Such a coupling results in some difficulties in the minimization process: one has not only to provide the orientation of the nematic vector $\n$ but as the direction of $\n$ varies, so does the shape of the surface changes and \emph{vice versa}.

In \cite{napoli2012surface, napoli2018influence}, Napoli and Vergori study the energy functional $\mathcal{E}(\n, S)$ when $\nabla=\nabla_S$ is the surface gradient defined as follows (see \cite{gurtin1975continuum}). Let $x \in S$ and let $\pi_x \colon \mathcal N_x\to S$ be the projection on $S$, where $\mathcal N_x$ is a neighborhood of $0$ on $T_xS$. Thus, $\nabla_S\n(x) := \nabla (\n \circ \pi_x)(0)$ and 
$$
|\nabla_S\n|^2 = \abs{\tens{D}\n}^2 + \abs{\tens{L}\n}^2,
$$
where $\tens{D}$ is the covariant derivative and $\tens{L}$ is the extrinsic curvature tensor of the surface $S$. More precisely, in \cite{napoli2018influence}, the case of revolution surfaces spanning two equal coaxial circles of radius $r$ placed at distance $2 h$ is considered. The Authors derive the Euler-Lagrange equations and study them numerically to visualize equilibrium configurations in terms of the parameter
$$
c := \frac{\kappa}{2 \gamma}.
$$
which affects the shape of the solutions (we refer to Figures 3 and 4 of \cite{napoli2018influence}).

Motivated by \cite{napoli2018influence}, we aim to rigorously study the equilibrium configurations and to analytically derive the existence of minimizers. 
However, as a simplification in studying this problem, we neglect the effects of the extrinsic curvature tensor, reducing the surface gradient to the covariant derivative $\tens{D}$ (we leave the case of the complete surface gradient for a subsequent analysis). Thus, our energy functional reads as
\begin{equation*}
    \mathcal{E}(\n, S)=\int_S \left(\gamma+\frac{\kappa}{2}|\tens{D}\n|^2\right)dA,
\end{equation*}
which has to be minimized among all the revolution surfaces spanning two equal coaxial circles of radius $r$ placed at distance $2 h$. The case of the covariant derivative has been investigated by Giomi in \cite{giomi2012hyperbolic}.
\textcolor{black}{In his paper, Giomi computes the first variation of the energy functional $\mathcal{E}$, getting
$$
\delta \mathcal{E} = \int_S \left(\gamma + \frac{\kappa}{2}K\right)\left(-2 H \eps\right)\, dA,
$$
where $K$ is the Gaussian curvature of $S$, $H$ is the mean curvature of $S$ and $\eps$ is a small displacement along normal direction to $S$. Thus, according to that formula, critical points of $\mathcal{E}$ are either minimal surfaces ($H = 0$) or surfaces with negative constant Gaussian curvature. We notice that, surfaces of the above types are not solutions of our Euler-Lagrange equation \eqref{EL}. As a consequence, Giomi's result seems to be in conflict with our Theorem \ref{main_total}.}

In order to find the explicit expression of $\mathcal E$, we parametrize $S$ by the profile curve $\rho\colon [-h,h]\to (0,+\infty)$ subjected to the constraint $\rho(\pm h)=r$, and we write the nematic director $\n=\cos \alpha \vect e_1+\sin \alpha \vect e_2$, where $\vect e_1$ is the direction of the parallel, while $\vect e_2$ is the meridian one. 
Thus, the minimization problem of $\E(\n, S)$ reduces (see Appendix \ref{sec-phy}) to the minimization of the purely geometric energy functional given by
\begin{equation*}
    \mathcal{E}(\rho)=\bigintsss_{-h}^h\left(\rho\sqrt{1+(\rho')^2}+ c\,\frac{(\rho')^2}{\rho\sqrt{1+(\rho')^2}} \right)\,dx.
\end{equation*}
If $c = 0$, then the problem reduces to the very classical problem of minimal surfaces of revolution (see the general reference \cite{Isenberg1978TheSO}). Nevertheless, we will investigate this case in detail since we need to set the problem in $W^{1,1}$ where much less is known about problems with linear growth. 
However, notice that the term 
\[
\frac{(\rho')^2}{\rho\sqrt{1+(\rho')^2}}
\]
is a sort of memory of the presence of the nematic director, which makes the study of $\E(\n, S)$ very interesting and which we are going to precisely describe below.

\subsection{Setting of the problem and main result}

Let $h,r>0$ and let 
$$
X:=\{\rho \in W^{1,1}(-h,h) : \rho> 0,\,\rho(-h)=\rho(h)=r\}.
$$
For all $c\ge 0$, let $\mathcal{E}_c \colon X \to \mathbb R$ be the functional given by 
\begin{equation}
    \label{energy_functional}
    \mathcal{E}_c(\rho)=\bigintsss_{-h}^h\left(\rho\sqrt{1+(\rho')^2}+ c\frac{(\rho')^2}{\rho\sqrt{1+(\rho')^2}}\right)\,dx.\ 
\end{equation}

Our main result concerns existence of minimizers of $\mathcal E_c$ on $X$ and qualitative properties of the minimizers. We now introduce some objects we will need in the statement of the main result. Let $\Xi$ be the unique positive solution of the equation 
\[
\xi \tanh \frac{1}{\xi}+\text{sech}^2\frac{1}{\xi}=\xi 
\]
and let 
\begin{equation}
    \label{eq:def_omega}
    \omega=\frac{1}{\Xi\cosh(1/\Xi)}. 
\end{equation}
An easy analysis shows that 
\begin{equation}\label{stima_omega}
\frac{52}{100}<\omega<\frac{53}{100}.
\end{equation}
This estimate will be useful later.
In what follows we will always assume that
\begin{equation}\label{cond_catenaria}
\frac{h}{r}\le \omega.
\end{equation}

Let $\rho_0 \in X$ be given by 
\begin{equation}
 \label{eq:catenaria}
    \rho_0(x)=\Pi_0\cosh\frac{x}{\Pi_0}.
\end{equation}
where $\Pi_0>0$ is the largest solution of the equation in $\Pi$ given by 
$$
    r = \Pi \cosh\frac{h}{\Pi}.
$$
Then, $\rho_0$ is the unique minimizer of $\E_0$ (see for details Theorem \ref{mainApp}). 

\bigskip

The main result of this paper is the following Theorem.

\begin{theorem}\label{main_total}
Assume \eqref{cond_catenaria}. For every $c\ge 0$ the functional $\mathcal E_c$ admits at least a minimizer. Furthermore, any minimizer $\rho_c$ of $\E_c$ satisfies the following properties:\begin{itemize}
\item[(a)] $\rho_c$ is even.
\item[(b)] $\rho_c$ is strictly convex.
\item[(c)] $\rho_0<\rho_c<r$ on $(-h,h)$ for all $c >0$.
\item[(d)] $\rho_c\to r$ uniformly on $[-h,h]$ as $c\to+\infty$.
\item[(e)] $\rho_c \in C^2([-h,h])$ and $\rho_c$ satisfies the Euler-Lagrange equation
    \begin{equation}\label{EL}
(1+(\rho')^2)((c+\rho^2) (\rho')^2+\rho^2)=\rho\rho'' (\rho^2(\rho')^2+c(2-(\rho')^2)+\rho^2)
    \end{equation}
    which has the first integral 
\begin{equation}
    \label{eq:integrale_primo}
    c\frac{(\rho_c')^2}{\rho_c(1+(\rho_c')^2)^{3/2}}-\frac{\rho_c}{\sqrt{1+(\rho_c')^2}}=-\rho_c(0).
\end{equation}
%where $\Pi_c>0$ is a positive constant which may depend on $c$.
\end{itemize} 
\end{theorem}
Numerical simulations of the corresponding revolution surfaces are presented in Figure \ref{fig:3dviews} to show some qualitatively properties stated in Theorem \ref{main_total}.

\begin{figure}[htbp]
\begin{subfigure}{.5\linewidth}
\centering
\includegraphics[width=0.85\textwidth]{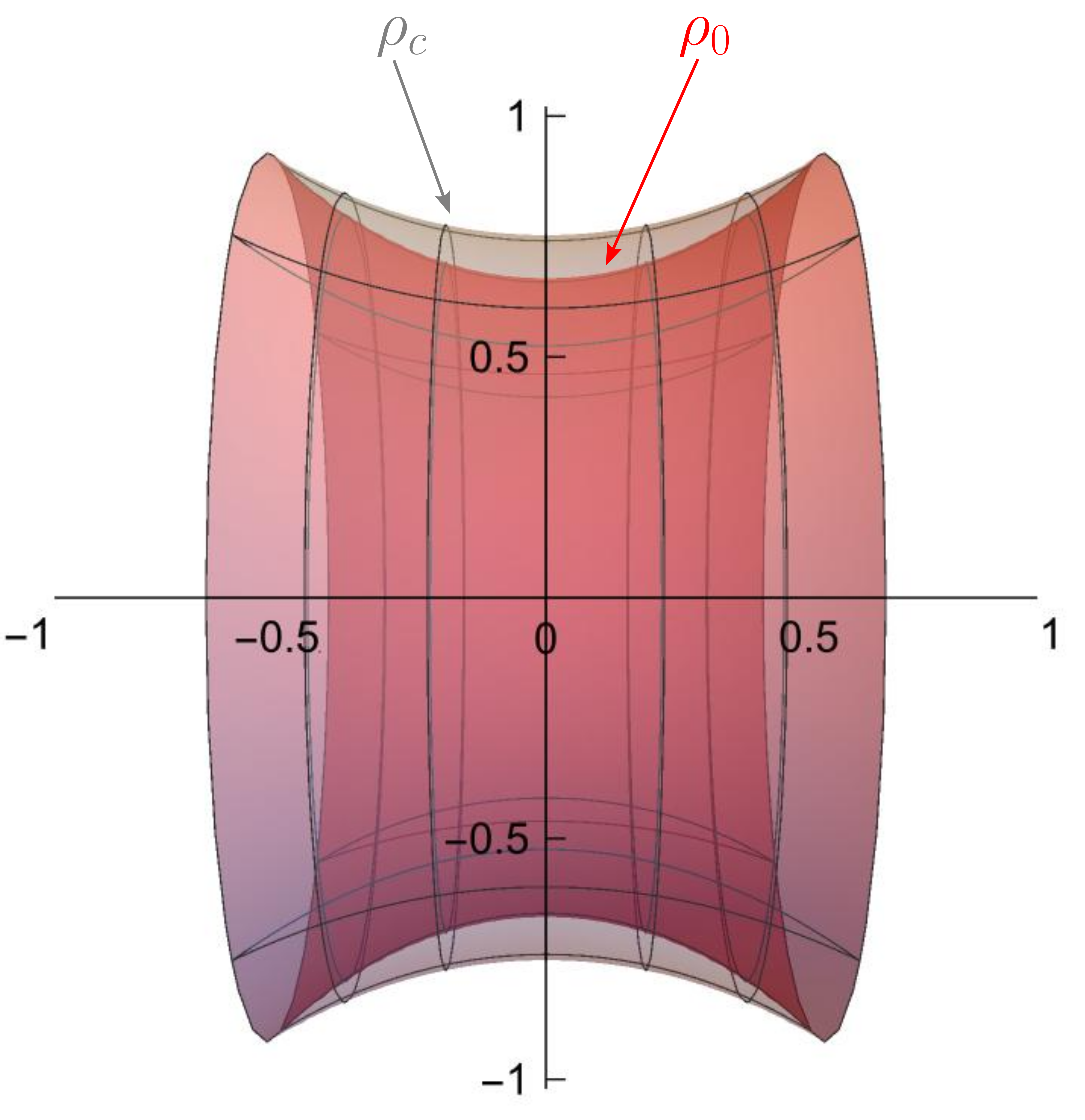}
\caption{Front view.}
\label{3d1}
\end{subfigure}%
\begin{subfigure}{.5\linewidth}
	\centering
 		\includegraphics[width=0.85\textwidth]{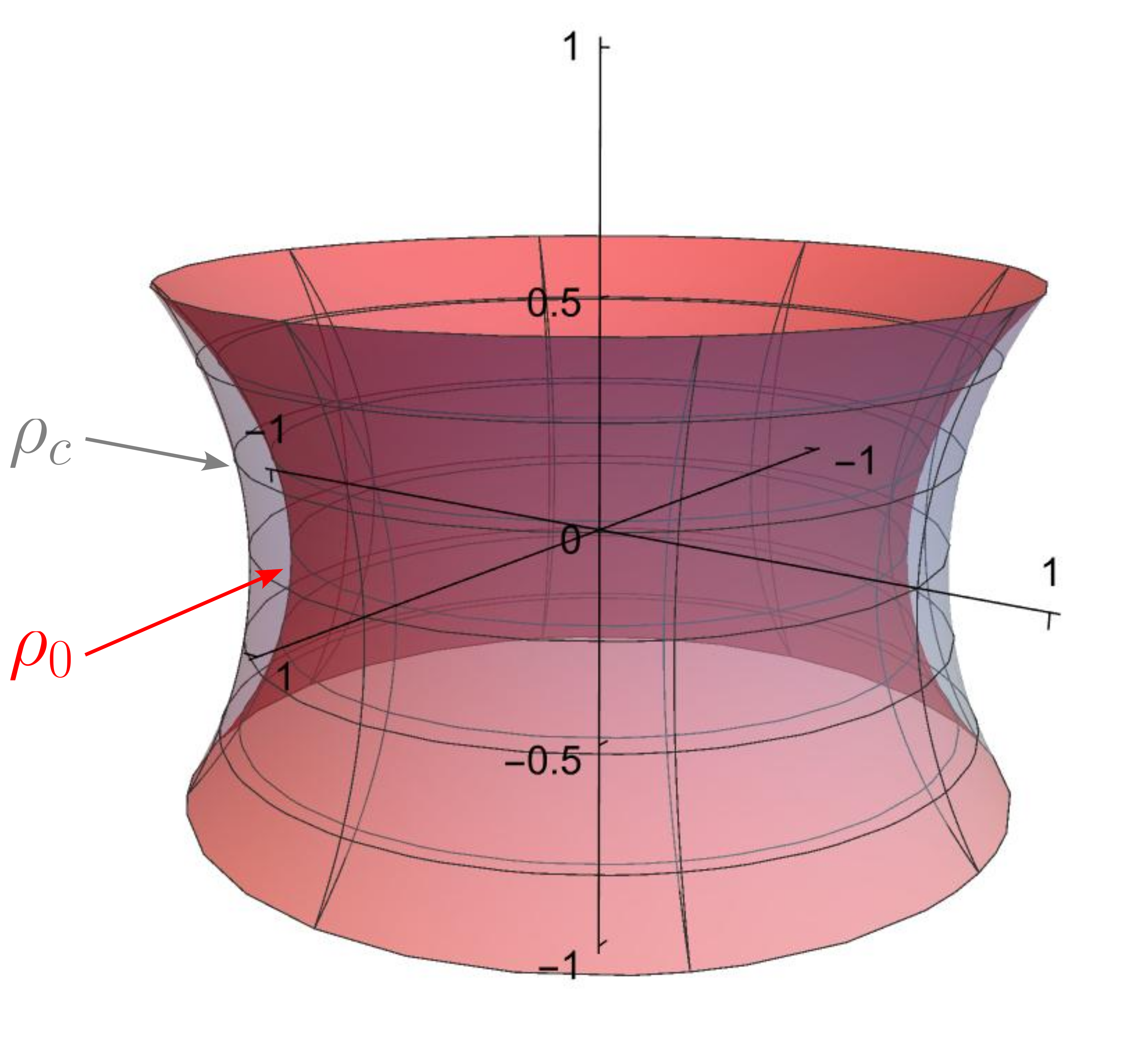}
\caption{3D view.}
\label{3d2}
\end{subfigure}
\caption{Three dimensional plots of $\rho_c$ solution of \eqref{EL} and $\rho_0$ solution of \eqref{eq:catenaria} having chosen $-h=-11/20, h=11/20$ and $r=9/10$. The red surface is the catenoid $\rho_0$ while the grey surface is the minimizer $\rho_c$ corresponding to the choice $c=0.1$.}
\label{fig:3dviews}
\end{figure}

\subsection{Strategy of the proof}
We decided to carry out the variational study in $W^{1,1}$, although the energy density has linear growth in $\rho'$, which 
would have suggested a more complicated treatment using BV functions. %To do this, we follow the paper by Greco \cite{greco2012} which investigates the area term by a suitable convex rearrangement argument and proves the existence of a minimizer in $W^{1,1}$.

The proof of existence in $W^{1,1}$ requires, as in \cite{greco2012}, a delicate analysis. A first key step (see Section \ref{sec_convex}) is to show that the energy decreases by convexification: this is proved for piecewise affine functions and then recovered for $W^{1,1}$ functions using a density argument. Since we have a minimizing sequence of convex functions, we can successfully apply Ascoli-Arzel\`a Theorem on every subinterval well contained in $[-h,h]$. The compactness on the whole interval $[-h,h]$ can be recovered whenever one proves that the derivative remains bounded close to the boundary. In order to obtain this last property, we need a complete and careful analysis of $\E_0$ which corresponds to the case of minimal surfaces of revolution. Precisely, we prove that if $h/r$ is small enough (see \eqref{cond_catenaria}), then the catenary $\rho_0$ is the unique minimizer of $\E_0$. When $c>0$, the necessary bound on the derivative holds true since for any $\rho \in X$ the function $\rho \vee \rho_0$ spends less energy. As a consequence, the slope of $\rho$ at the boundary is controlled by the slope of the catenary $\rho_0$. Therefore, the uniform convergence on $[-h,h]$ of a minimizing sequence guarantees that boundary conditions pass to the limit. Moreover, it is easy to see that the energy functional is continuous with respect to the $W^{1,1}$ convergence. Thus, the Direct method of Calculus of Variations ensures the existence of a minimizer $\rho_c$.

Whenever a minimizer exists, we prove several properties of it. The weak form of the Euler-Lagrange equation is derived by a standard first variation argument. The strict convexity of $\rho_c$ follows from the strong form of the Euler-Lagrange equation: it is crucial to have a fine control on $\rho_c'$. This leads to the $C^2$ regularity of $\rho_c$ up to the boundary as well as the derivation of the first integral. In particular, we get the symmetry of $\rho_c$ by means of an ODE argument. Another consequence of the strong form of the Euler-Lagrange equation is the strict inequality $\rho_0<\rho_c$ on $(-h,h)$, deduced by a contradiction argument. Finally, using essentially a $\Gamma$-convergence technique, we get the asymptotics as $c\to+\infty$ proving that $\rho_c$ converges uniformly to a constant function.

\bigskip

The plan of the paper is the following. The form of the energy functional in the revolution surface setting is completely derived in the Appendix \ref{sec-phy}. In Section \ref{sec_convex}, we provide the convexification argument. In Section \ref{sec:E0}, we carry out a complete treatment of $\E_0$, while in Section \ref{sec:esistenza_minimi_Ec} we prove the existence of minimizers for the energy functional $\E_c$. Finally, Section \ref{sec:proprieta_geometriche} is devoted to derive some qualitative properties of minimizers.

\subsection{Further remarks and future investigations}
We run some numerical tests in order to understand better the behaviour of the solutions $\rho_c$, like their dependence on the constant $c$. Indeed, differently from the case $c= 0$, where we know exactly the shape of the unique minimizer by integration, for $c>0$ the first integral \eqref{eq:integrale_primo} cannot be  integrated explicitly. Thus, we numerically solve the Euler-Lagrange equation \eqref{EL} in the regime where \eqref{cond_catenaria} holds true in order to have the catenary $\rho_0$ as the unique minimizer for $c = 0$.
\begin{figure}[htbp]
\begin{subfigure}{.5\linewidth}
	\centering
	\includegraphics[width=0.98\textwidth]{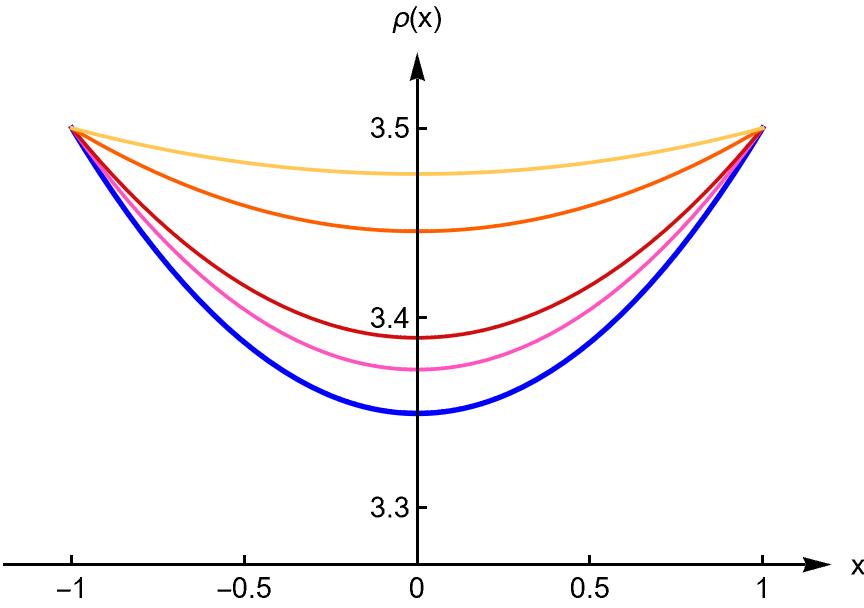}
	\caption{$\rho(1) = \rho(-1) = 7/2$}
	\label{ris1}
\end{subfigure}%
\begin{subfigure}{.5\linewidth}
	\centering
	\includegraphics[width=0.98\textwidth]{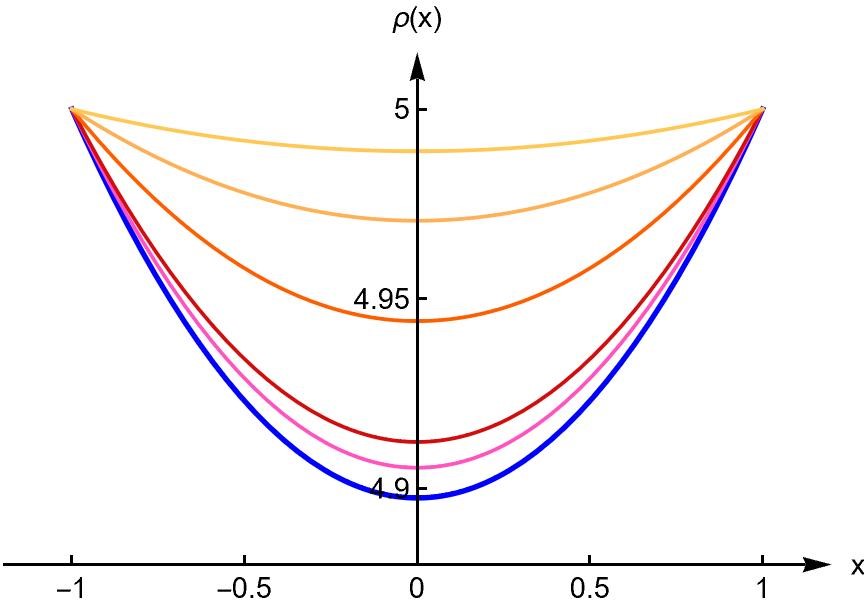}
	\caption{$\rho(1) = \rho(-1) = 5$}
	\label{ris2}
	\end{subfigure}
	\caption{(a) Numerical solutions of \eqref{EL} having chosen $-h= -1$, $h = 1$, as boundary conditions $\rho(1) = \rho(-1) = 7/2$ and $c = 0, 1,2,10,30$. The thicker blue line is the catenary for $c = 0$. (b) Numerical solutions of \eqref{EL} having chosen $-h= -1$, $h = 1$, as boundary conditions $\rho(1) = \rho(-1) = 5$ and $c = 0, 1,2,10,30,100$. The thicker blue line is the catenary for $c = 0$.}
	\label{fig:numerica}
\end{figure}
In Figure \ref{fig:numerica}, we plot numerical solutions $\rho_c$ of \eqref{EL} varying the constant $c \in [0, +\infty)$. As expected from Theorem \ref{main_total}, we notice that the catenary $\rho_0$ (thick blue line in Figure \ref{fig:numerica}) is a barrier from below for all $c >0$.
%and $\forall\, x \in (-h,h)$ we have $\rho_0(x) < \rho_c(x)$. 
Moreover, for large values of $c$ (see Figure \ref{ris2}), plots are in accordance with Theorem \ref{main_total}: the solution becomes flatter, confirming that for $c\to +\infty$, the limit $\rho_\infty$ is constant function and the associated revolution surface is a cylinder of radius $r$ and height $2 h$. 

From these numerical plots, we can obtain additional interesting geometrical properties which we are not able to prove rigorously. For instance, directly from Figure \ref{fig:numerica}, we notice that solutions are ordered with respect to the parameter $c$: when $c$ increases, so does $\rho_c(0)$, suggesting a monotonicity of $\rho_c$ with respect to the parameter $c$. Moreover, similarly to what happens for the case $c = 0$, we expect the minimizer to be unique. Precisely, for all $c >0$, there might be two families of solutions of the first integral: one is close to the stable catenary $\rho_0$ and it is what numerical results should find, while the other one should be close to the unstable catenary $\rho_1$.

In our analysis we do not consider any boundary condition on $\n$, namely on the angle $\alpha$. In this regard, notice that if we require that $\alpha$ satisfies some boundary conditions, then the derivatives of $\alpha$ cannot be neglected in \eqref{eq:simpl_funzionale}. This aspect could be an interesting future investigation. An ongoing project is the rigorous proof of the results of Napoli and Vergori \cite{napoli2018influence} where the covariant derivative is replaced by the surface gradient in the setting of revolution surfaces.

\section{A convexification argument}\label{sec_convex}

In what follows $\rho^\ast$ denotes the convex envelope of $\rho$, namely the greatest convex function lower than $\rho$. We observe that by construction $\rho^\ast$ remains in $X$ whenever $\rho \in X$. The main result of this section is the following Proposition.

\begin{proposition}\label{lemmaconvexity-affine_con}
For any $c \geq 0$ and for any $\rho\in X$, we have 
\begin{equation}\label{conv0}
\mathcal E_c(\rho^\ast)\le \mathcal E_c(\rho).
\end{equation}
\end{proposition}

To prove Proposition \ref{lemmaconvexity-affine_con}, we first consider piecewise affine functions. Precisely, we only show the validity of \eqref{conv0} in a particular context represented in Figure \ref{fig:piecewise}.
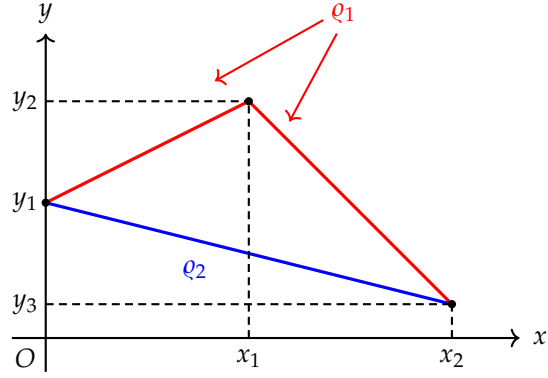
\begin{figure}[htbp]
		\centering
		\begin{tikzpicture}[rotate=0, scale= 0.9]
\coordinate (O) at (0,0);
\draw[thick, color=black,->] (-0.5,0) -- (7,0);
\draw[thick, color=black,->] (0,-0.5) -- (0,4.5);
\draw[very thick, color=red] (0,2) -- (3,3.5);
\draw[very thick, color=red] (3,3.5) -- (6,0.5);
\draw[very thick, color=blue] (0,2) -- (6,0.5);
\draw[thick, densely dashed, color=black] (3,0) -- (3,3.5);
\draw[thick, densely dashed, color=black] (6,0) -- (6,0.5);
\draw[thick, densely dashed, color=black] (0,0.5) -- (6,0.5);
\draw[thick, densely dashed, color=black] (0,3.5) -- (3,3.5);
\filldraw [black] (0,2) circle (1.5pt);
\filldraw [black] (3,3.5) circle (1.5pt);
\filldraw [black] (6,0.5) circle (1.5pt);
%\draw node at (5.3,2.5) {\textcolor{blue}{$\rho_1$}};
\draw node at (2.2,1) {\textcolor{blue}{$\rho_2$}};
\draw node at (-0.3,-0.3) {\textcolor{black}{$O$}};
\draw node at (0,4.8) {\textcolor{black}{$y$}};
\draw node at (7.3,0) {\textcolor{black}{$x$}};
\draw node at (6,-0.3) {\textcolor{black}{$x_2$}};
\draw node at (3,-0.3) {\textcolor{black}{$x_1$}};
\draw node at (-0.3,2) {\textcolor{black}{$y_1$}};
\draw node at (-0.3,3.5) {\textcolor{black}{$y_2$}};
\draw node at (-0.3,0.5) {\textcolor{black}{$y_3$}};
\draw[thick, color=red,<-] (3.6,3.2) -- (4.3,4.5);
\draw[thick, color=red,<-] (2.5,3.8) -- (4.1,4.7);
\draw node at (4.4,4.8) {\textcolor{red}{$\rho_1$}};
\end{tikzpicture}
	\caption{Graphical representation of the two piecewise affine functions $\rho_1$ and $\rho_2$ defined in \eqref{eq:rho1} and \eqref{eq:rho2}.}
	\label{fig:piecewise}
\end{figure}
All the other cases can be easily treated in the same way using reflection arguments. Let $x_2>0$ and $\rho\in W^{1,1}(0,x_2)$, we introduce 
\[
\begin{aligned}
    &\mathcal A(\rho)=\int_0^{x_2}\rho\sqrt{1+(\rho')^2}\,dx, \qquad &&\mathcal N(\rho)=\int_0^{x_2}\frac{(\rho')^2}{\rho\sqrt{1+(\rho')^2}}\,dx.
\end{aligned}
\]
Fix $x_1\in (0,x_2)$, $y_1,y_2,y_3\in \R$ with $y_3<y_1\le y_2$ and $y_1\ge 0$. Let $\rho_1,\rho_2 \colon [0,x_2]\to \R$ be two functions respectively given by 
\begin{equation}
    \label{eq:rho1}
    \rho_1(x)=\left\{
    \begin{array}{ll}
\displaystyle \frac{y_2-y_1}{x_1}x+y_1 & \text{if $x\in [0,x_1]$}\\
\\
\displaystyle\frac{y_3-y_2}{x_2-x_1}(x-x_1)+y_2 & \text{if $x\in (x_1,x_2]$}
\end{array}\right.
\end{equation}
and 
\begin{equation}
    \label{eq:rho2}
    \rho_2(x)=\frac{y_3-y_1}{x_2}x+y_1.
\end{equation}

\begin{lemma}\label{lemmaconvexity-affine}
Using the same notations as above, it holds 
\begin{equation}\label{area_conv}
\mathcal A(\rho_2)\le \mathcal A(\rho_1).
\end{equation}
\end{lemma}

\begin{proof}
{\it Step 1}. Assume $y_1=y_2=0$. In this case, since $y_3<0$, the inequality \eqref{area_conv} reduces to 
\[
\sqrt{x_2^2+y_3^2}\ge \sqrt{(x_2-x_1)^2+y_3^2},
\]
which is true by construction. 
\\
\\
{\it Step 2}. Assume that $y_1=0<y_2$. Let $\overline x\in (x_1,x_2)$ be such that $\rho_1(\overline x)=0$ and let $\rho_3\colon [0,x_2]\to \R$ be given by 
\[
\rho_3(x)=\left\{\begin{array}{ll}
0 & \text{if $x\in [0,\overline x]$},\\
\\
\rho_1(x) & \text{if $x\in (\overline x,x_2]$}.
\end{array}\right.
\]
Then again
\[
\mathcal A(\rho_2)\stackrel{\text{Step 1}}{\le}\mathcal A(\rho_3)=\int_{\overline x}^{x_2}\rho_1\sqrt{1+(\rho_1')^2}\,dx\le \mathcal A(\rho_1).
\]
{\it Step 3}. Assume that $0<y_1\le y_2$. By Step 1 and Step 2, we have $\mathcal A(\rho_2-y_1)\le \mathcal A(\rho_1-y_1)$. Hence 
\[
\begin{aligned}
\mathcal A(\rho_2)&=\mathcal A(\rho_2-y_1)+y_1\int_0^{x_2}\sqrt{1+(\rho_2')^2}\,dx\\
&\le \mathcal A(\rho_1-y_1)+y_1\int_0^{x_2}\sqrt{1+(\rho_1')^2}\,dx\\
&=\mathcal A(\rho_1)
\end{aligned}
\]
and this concludes the proof.
\end{proof}

\begin{lemma}\label{lemmaconvexity-affine1}
Assume $y_3>0$. Then 
\begin{equation}\label{nematico_conv}
\mathcal N(\rho_2)\le \mathcal N(\rho_1).
\end{equation}
\end{lemma}

\begin{proof}
The proof proceeds as in the proof of Lemma \ref{lemmaconvexity-affine}.
\\
\\
{\it Step 1}. Assume $y_1=y_2$. By a straightforward computation, \eqref{nematico_conv} reduces to 
\[
\frac{1}{\sqrt{x_2^2+(y_3-y_1)^2}}\le \frac{1}{\sqrt{(x_2-x_1)^2+(y_3-y_1)^2}},
\]
which is true by construction. 
\\
\\
{\it Step 2}. Assume that $y_1<y_2$. Let $\overline x\in (x_1,x_2)$ be such that $\rho_1(\overline x)=y_1$ and let $\rho_3\colon [0,x_2]\to \R$ be given by 
\[
\rho_3(x)=\left\{\begin{array}{ll}
y_1 & \text{if $x\in [0,\overline x]$}\\
\\
\rho_1(x) & \text{if $x\in (\overline x,x_2]$}.
\end{array}\right.
\]
Then again
\[
\mathcal N(\rho_2)\stackrel{\text{Step 1}}{\le}\mathcal N(\rho_3)=\int_{\overline x}^{x_2}\frac{(\rho_1')^2}{\rho_1\sqrt{1+(\rho_1')^2}}\,dx\le \mathcal N(\rho_1)
\]
and this yields the conclusion.
\end{proof}

The next Lemma is fundamental to ensure that the piecewise approximation of the convex envelope strongly converges in $W^{1,1}$. We present it without proof since it is exactly Lemma 3.2 in \cite{greco2012}.

\begin{lemma}\label{lemma_conv}
Let $a,b\in \R$ with $a<b$ and let $(\rho_j)$ be a sequence of convex functions on $(a,b)$ such that $\rho_j\to \rho$ uniformly on every $[\gamma,\delta]\subset (a,b)$. Then $\rho_j' \to \rho'$ in $L^1(\gamma,\delta)$ for any $[\gamma,\delta]\subset (a,b)$.
\end{lemma}

Once derivatives converge, we need to ensure that the nematic contribution is continuous in the $L^1$ topology.

\begin{lemma}\label{lemma_cont}
Let $a,b\in \R$ with $a<b$ and let $(u_j)$ be a sequence in $L^1(a,b)$ with $u_j\to u$ strongly in $L^1(a,b)$. Then 
\begin{equation}\label{converg1}
\sqrt{1+u_j^2}\to \sqrt{1+u^2} \quad \text{in $L^1(a,b)$}
\end{equation}
and 
\begin{equation}\label{converg2}
\frac{u_j^2}{\sqrt{1+u_j^2}}\to \frac{u^2}{\sqrt{1+u^2}} \quad \text{in $L^1(a,b)$}.
\end{equation}
\end{lemma}

\begin{proof}
Concerning the proof of \eqref{converg1}, using the subadditivity of the square root we have
\[
\left|\sqrt{1+u_j^2}- \sqrt{1+u^2}\right|\le \sqrt{|u_j^2-u^2|}. 
\]
Hence, by H\"older's inequality, we obtain 
\[
\begin{aligned}
\int_a^b\left|\sqrt{1+u_j^2}- \sqrt{1+u^2}\right|\,dx&\le\int_a^b \sqrt{|u_j^2-u^2|}\,dx\\
&=\int_a^b \sqrt{|u_j+u|}\sqrt{|u_j-u|}\,dx\\
&\le \left(\int_a^b|u_j+u|\,dx\right)^{1/2}\left(\int_a^b|u_j-u|\,dx\right)^{1/2}\\
&\le C \left(\int_a^b|u_j-u|\,dx\right)^{1/2}\to 0.
\end{aligned} 
\]
Next, in order to show \eqref{converg2}, we first introduce the function
\[
\ell(x):=\left\{\begin{aligned}
    &\R \to \R\\
&x\mapsto \frac{x^2}{\sqrt{1+x^2}}
\end{aligned}
\right.
\]
and we observe that it is Lipschitz. Denoting by $L$ its Lipschitz constant, we obtain
\[
\bigintsss_a^b \left|\frac{u_j^2}{\sqrt{1+u_j^2}}- \frac{u^2}{\sqrt{1+u^2}} \right|\,dx=\int_a^b|\ell(u_j)-\ell(u)|\,dx\le L\int_a^b|u_j-u|\,dx\to 0
\]
and this yields the conclusion.
\end{proof}

Thus, we are in position to show Proposition \ref{lemmaconvexity-affine_con}.

\begin{proof}[Proof of Proposition \ref{lemmaconvexity-affine_con}]
    Let $(\rho_j)$ be a sequence of piecewise affine functions in $X$ with $\rho_j \to \rho$ in $W^{1,1}$. Then $\rho_j \to \rho$ uniformly on $[-h,h]$. It is easy to see that $\rho_j^\ast$ remains piecewise affine and $\rho_j^\ast \to \rho^\ast$  uniformly on $[-h,h]$ as well. Applying Lemma \ref{lemma_conv} we obtain $(\rho_j^\ast)' \to (\rho^\ast)'$ in $L^1(-h,h)$. As a consequence we deduce that $\rho_j^\ast \to \rho^\ast$ in $W^{1,1}$. By Lemma \ref{lemmaconvexity-affine} and Lemma \ref{lemmaconvexity-affine1} we have $\mathcal E_c(\rho_j^\ast)\le \mathcal E_c(\rho_j)$. The conclusion follows using the continuity in Lemma \ref{lemma_cont} and passing to the limit.
\end{proof}

\section{Existence of a unique minimizer of \texorpdfstring{$\E_0$}{E0}}
\label{sec:E0}

In this section, we examine carefully the case $c = 0$. In order to show the existence of a minimizer for $\mathcal{E}_0$, we need to introduce an auxiliary functional. We define 
$$
X_0=\{\rho\in W^{1,1}(-h,h): \rho\ge0,\,\rho(-h)\le r,\,\rho(h)\le r\}
$$ 
and the energy functional 
\begin{equation}\label{tildeE}
\widetilde\E_0\colon = 
\left\{
\begin{aligned}
  &X_0 \to \R \\
  &\rho\mapsto \bigintsss_{-h}^h\rho\sqrt{1+(\rho')^2}\,dx+r^2-\frac{\rho(-h)^2+\rho(h)^2}{2}.
\end{aligned}
\right.
\end{equation}

\begin{remark}
Following \cite{bm1991}, it is possible to prove that $\widetilde\E_0$ is the restriction to $X_0$ of the relaxation of the functional 
\[
W^{1,1}(-h,h)\ni \rho \mapsto \left\{\begin{array}{ll}\E_0(\rho) & \text{if $\rho\in X_0$}\\
\\
+\infty & \text{otherwise} 
\end{array}\right.
\]
with respect to the local uniform convergence. Nevertheless, we do not need to prove that. Indeed, in the next Proposition we will adapt the proof of \cite[Theorem 4.2]{greco2012} showing directly the continuity of $\widetilde\E_0$ among all sequences of convex functions converging locally uniformly. 
\end{remark}

\begin{proposition}\label{cont_greco}
Let $(\rho_j)$ be a sequence of convex functions in $X_0$. Assume that $\rho_j\to \rho$ as $j\to +\infty$ uniformly on any closed interval $[\gamma,\delta]\subset (a,b)$ for some $\rho\in X_0$. Then
\begin{equation}\label{contE}
\lim_{j\to+\infty}\widetilde \E_0(\rho_j)=\widetilde \E_0(\rho).
\end{equation}
\end{proposition}

\begin{proof}
{\it Step 1}. We first claim that for $\eta>0$ small enough it holds 
\begin{equation}\label{lim1}
\lim_{j\to+\infty}\int_{-h+\eta}^{h-\eta}\rho_j\sqrt{1+(\rho_j')^2}\,dx=\int_{-h+\eta}^{h-\eta}\rho\sqrt{1+(\rho')^2}\,dx.
\end{equation}
Indeed, \eqref{lim1} holds true since it is sufficient to combine the locally uniform convergence of  $\rho_j$ to $\rho$ with Lemma \ref{lemma_conv} and Lemma \ref{lemma_cont}.
\\
\\
{\it Step 2}. Assume that $\rho_j(-h)=r$, $\rho_j(h)=r$, $\rho(-h)=r$ and $\rho(h)=r$. We are going to prove that if $\sigma>0$ and $\eta>0$ is such that $|\rho(x)-r|<\sigma/2$ on $(-h,-h+\eta)\cup (h-\eta,h)$, then for some $M\ge 1$, it holds
\begin{equation}\label{lim2}
\left|\int_{-h}^{-h+\eta}\rho_j\sqrt{1+(\rho_j')^2}\,dx\right|\le M\int_{r-\sigma}^{r}y\,dy
\end{equation}
and
\begin{equation}\label{lim3}
\left|\int_{h-\eta}^h\rho_j\sqrt{1+(\rho_j')^2}\,dx\right|\le M\int_{r-\sigma}^ry\,dy.
\end{equation}
Combining \eqref{lim2},\eqref{lim3} with \eqref{lim1}, we easily deduce \eqref{contE}. For instance, let us show \eqref{lim2}. Since $\rho_j$ are convex functions and $\rho_j \to \rho$ locally uniformly on $(-h,h)$, we deduce that $\rho_j$ is invertible on $(-h,-h+\eta)$. Denoting by $u_j$ the inverse and making the change of variable $y=\rho_j(x)$ we obtain 
\[
\int_{-h}^{-h+\eta}\rho_j\sqrt{1+(\rho_j')^2}\,dx=\int_{\rho_j(-h+\eta)}^ry\sqrt{1+(u_j')^2}\,dy.
\]
Since $u_j'$ is uniformly bounded on $(\rho_j(-h+\eta),r)$, we get
\[
\left|\int_{\rho_j(-h+\eta)}^ry\sqrt{1+(u_j')^2}\,dy\right|\le M\int_{r-\sigma}^ry\,dy
\]
for some $M\ge 1$, which gives \eqref{lim2}.
\\
\\
{\it Step 3.} Assume now that $\rho_j(-h)=r$, $\rho_j(h)=r$, $\rho(-h)<r$ and $\rho(h)<r$ (the case $\rho(-h)=r$ and $\rho(h)<r$ and the case $\rho(-h)<r$ and $\rho(h)=r$ are similar). We divide the interval $(-h,-h+\eta)$ in a suitable $a_j$ and the interval $(h-\eta,h)$ in a suitable $b_j$. Precisely, take $\sigma>0$ with $\sigma<r-\rho(-h)$ and with $\sigma<r-\rho(h)$. 
As before, take $\eta>0$ such that $|\rho(x)-\rho(-h)|<\sigma/2$ on  $(-h,-h+\eta)$ and $|\rho(x)-\rho(h)|<\sigma/2$ on $(h-\eta,h)$. Let $a_j \in (-h,-h+\eta)$ be the unique solution of $\rho_j(a_j)=\rho(-h)+\sigma$ and let $b_j \in (h-\eta,h)$ be the unique solution of $\rho_j(b_j)=\rho(h)+\sigma$. Using the same argument as in Step 2, we can conclude. Indeed, $\rho_j$ is invertible on $(-h,a_j)$ and denoting by $u_j$ the inverse and making the change of variable $y=\rho_j(x)$, we obtain 
\[
\int_{-h}^{a_j}\rho_j\sqrt{1+(\rho_j')^2}\,dx=\int^r_{\rho(-h)+\sigma}y\sqrt{1+(u_j')^2}\,dy.
\]
In this case $u_j'\to 0$ uniformly on $(\rho(-h)+\sigma,r)$. As a consequence
\[
\lim_{j\to+\infty}\int^r_{\rho(-h)+\sigma}y\sqrt{1+(u_j')^2}\,dy=\int_{\rho(-h)+\sigma}^{r}y\,dy=\frac{r^2}{2}-\frac{(\rho(-h)+\sigma)^2}{2}.
\]
Similarly, 
\[
\lim_{j\to+\infty}\int_{b_j}^h\rho_j\sqrt{1+(\rho_j')^2}\,dx=\frac{r^2}{2}-\frac{(\rho(h)+\sigma)^2}{2}.
\]
Using again \eqref{lim1} and the arbitrariness of $\sigma$, we deduce \eqref{contE}.
\\
\\
{\it Step 4.} The general case can be easily obtained by a density argument.
\end{proof}

In what follows we will often use the function 
\begin{equation}\label{Phi}
\Phi\colon \left\{
\begin{aligned}
     (0,+\infty) &\to (0,+\infty)\\
     x &\mapsto \frac{\cosh x}{x}.
\end{aligned}
\right. 
\end{equation}
We also denote 
\begin{equation}
    \label{min_phi}
    m:=\min\Phi.
\end{equation}

\begin{theorem}\label{mainApp}
Assume \eqref{cond_catenaria}. The catenary $\rho_0$ defined in \eqref{eq:catenaria} is the unique minimizer of $\E_0$.
\end{theorem}

\begin{proof} 
{\it Step 1.} We prove that $\widetilde\E_0$ has a even and convex minimizer $\rho_m$ on $X_0$. Let $(\rho_j)$ be a minimizing sequence. First of all, we show that we can assume $\rho_j$ even. Indeed, let
\[
a_j=\int_{-h}^0\rho_j\sqrt{1+(\rho_j')^2}\,dx+r^2-\frac{\rho_j(-h)^2+\rho_j(h)^2}{2}, \qquad b_j=\int_0^h\rho_j\sqrt{1+(\rho_j')^2}\,dx+r^2-\frac{\rho_j(-h)^2+\rho_j(h)^2}{2}.
\]
If $a_j=b_j$ then $\widetilde \E_0(\rho_j)=\widetilde \E_0(\rho_j(|\cdot|))$. Assume $a_j\ne b_j$, for instance $a_j<b_j$. Then, taking $\overline\rho_j(x)=\rho_j(x)$ if $x\in (-h,0]$ and $\overline \rho_j(x)=\rho_j(-x)$ if $x\in [0,h)$, we get $\widetilde \E_0(\overline\rho_j)< \widetilde \E_0(\rho_j)$. This means that $(\overline\rho_j)$ is still a minimizing sequence for $\widetilde \E_0$. 

From now on, we therefore assume $\rho_j$ even. Notice that $\rho_j^\ast \in X_0$ is even as well. Since by construction $\rho_j(\pm h)=\rho_j^\ast(\pm h)$, using Proposition \ref{lemmaconvexity-affine_con}, we get 
\[
\widetilde \E_0(\rho_j^\ast)=\E_0(\rho_j^\ast)+r^2-\rho_j^\ast(h)^2\le \E_0(\rho_j)+r^2-\rho_j(h)^2=\widetilde \E_0(\rho_j).
\]
As a consequence, the sequence $(\rho_j^\ast)$ is still a minimizing sequence for $\widetilde \E_0$. By convexity $\rho_j^\ast \le r$ everywhere, so that the sequence $(\rho_j^\ast)$ is uniformly bounded. Since $\rho_j^\ast$ are convex, we can apply the Ascoli-Arzel\`a Theorem on any closed interval $[a,b]\subset (-h,h)$. As a consequence, we deduce that there is $\rho_m\in X_0$ such that (up to a subsequence not relabeled) $\rho_j^\ast\to \rho_m$ uniformly on $[a,b]$. By \eqref{contE}, the Direct method of the Calculus of Variations ensures that $\rho_m$ is a minimizer of $\widetilde\E_0$.
Moreover, $\rho_m$ is even and convex since it is a pointwise limit of even and convex functions. In particular, $\rho_m$ is decreasing on $[-h,0]$ and increasing on $[0,h]$.

In what follows, we are going to prove that such a limit $\rho_m$ coincides with $\rho_0$. This implies that $\rho_m\in X$ and therefore $\rho_m$ minimizes $\mathcal E_0$ on $X$. From now on, for simplicity of notation, we denote $\rho=\rho_m$.
\\
\\
{\it Step 2.} Assume there is an open interval $(a,b) \subseteq [0,h]$ where $\rho>0$. We claim that
\begin{equation}\label{forma}
\rho(x)=\Pi\cosh\frac{x-x_0}{\Pi}, \quad x\in (a,b),
\end{equation}
for some $x_0\in \R$ and some $\Pi>0$. In particular, $\rho \in C^1(a,b)$. 

Let $\varphi\in C^1(-h,h) \cap C^1_c(a, b)$ and let $\sigma>0$ be small enough. Then, for any $t\in (-\sigma,\sigma)$ we have 
\[
0=\frac{d}{dt}\widetilde \E_0(\rho+t\varphi)_{\big|_{t=0}}=\int_a^b\sqrt{1+(\rho')^2}\varphi\,dx+\int_a^b\frac{\rho\rho'}{\sqrt{1+(\rho')^2}}\varphi'\,dx.
\]
Let $F \colon [0,h] \to \R$ be given by  
\[
F(x)=\int_0^x\sqrt{1+(\rho')^2}\,dx,
\]
then 
\[
\bigintsss_a^b \left(\frac{\rho\rho'}{\sqrt{1+(\rho')^2}}-F\right)\varphi'\,dx=0.
\]
Applying the Du Bois-Reymond Lemma, we deduce that 
\begin{equation}\label{el_weak}
\frac{\rho\rho'}{\sqrt{1+(\rho')^2}}-F=k, \qquad \text{a.e.\,on $(a,b)$}.
\end{equation}
where $k$ is a constant. Let us introduce the function
\[
\Psi:=\left\{\begin{aligned}
    &[0,+\infty) \to\R\\
    &x \mapsto\frac{x}{\sqrt{1+x^2}}.
\end{aligned}
\right.
\]
Then, $\Psi^{-1} \in C^1([0,1))$ and 
\[
\rho'=\Psi^{-1}\left(\frac{F+k}{\rho}\right),\qquad \text{a.e.\,on $(a,b)$}.
\]
As a consequence, $\rho'\in W^{1,1}(a,b)$ which means $\rho\in W^{2,1}(a,b)$. Differentiating \eqref{el_weak}, we get 
\[
\rho\rho''=1+(\rho')^2, \qquad \text{a.e.\,on $(a,b)$},
\]
which has the prime integral 
\[
\frac{\rho}{\sqrt{1+(\rho')^2}}=\Pi,\qquad \text{a.e.\,on $(a,b)$}
\]
where $\Pi$ is a constant. We therefore obtain $\Pi^2\rho''=\rho$. It is classically known that the general solution can be written as 
\[
\rho(x)=\Pi\cosh\frac{x-x_0}{\Pi}, \quad x\in (a,b),
\]
for some $x_0\in \R$.
\\
\\
{\it Step 3.} We prove that if $\rho(h)>0$ then $\rho(h)=r$. Since $\rho$ is even, it follows also that if $\rho(-h)>0$ then $\rho(-h)=r$. Assume by contradiction that $\rho(h)\in (0,r)$. Let $\eta \in (0,h)$ and let $\zeta \in (\rho(h),r)$. Let $\rho_\eta\in Y$ be given by 
\[
\rho_\eta(x)=\left\{\begin{array}{ll}
\rho(x)& \text{if $x\in (-h,h-\eta)$}\\
\\
\displaystyle\frac{\zeta-\rho(h-\eta)}{\eta}(x-h)+\zeta& \text{if $x\in [h-\eta,h)$}
\end{array}\right.
\]
and let $g\colon(0,h)\to \R$ be given by $g(\eta)=\widetilde \E_0(\rho_\eta)$. Since $\rho_\eta \to \rho$ uniformly on any $[a,b]\subset (-h,h)$, using \eqref{contE}, we deduce that 
\[
\lim_{\eta\to0^+}g(\eta)=\widetilde \E_0(\rho).
\]
As a consequence, $g$ becomes continuous on $[0,h)$ with $g(0)=\widetilde \E_0(\rho)$. Moreover, $g$ is differentiable on $(0,h)$. By definition, for any $\eta>0$ we have 
\begin{equation}\label{E}
g(\eta)=\underbrace{\int_{-h}^{h-\eta}\rho\sqrt{1+(\rho')^2}\,dx}_{\mathcal{T}_1}+\underbrace{\int_{h-\eta}^h\rho_\eta\sqrt{1+(\rho_\eta')^2}\,dx}_{\mathcal{T}_2}+r^2-\frac{\rho(-h)^2+\zeta^2}{2}.
\end{equation}
Differentiating $g(\eta)$ with respect to $\eta$ and computing in $\eta= 0$, noticing that $g'(\eta) = \partial_\eta \mathcal{T}_1 + \partial_\eta \mathcal{T}_2$, we get for the first term $\mathcal{T}_1$ in \eqref{E}
\[
\frac{d\mathcal{T}_1}{d\eta}_{\big|_{\eta=0}}=-\rho(h)\sqrt{1+(\rho'(h))^2},
\]
where $\rho'(h)$ exists by convexity and it is finite by Step 1, since $\rho>0$ in a left neighborhood of $h$. A similar computation shows that for $\mathcal{T}_2$ we have
\[
\mathcal{T}_2 =\int_{h-\eta}^h\rho_\eta\sqrt{1+(\rho_\eta')^2}\,dx=\sqrt{\eta^2+(\zeta-\rho(h-\eta))^2}\,\frac{\zeta+\rho(h-\eta)}{2}.
\]
As a consequence, computing the limit as $\eta \to 0^+$, we obtain
\[
\frac{d\mathcal{T}_2}{d\eta}_{\big|_{\eta=0}}=\rho(h)\rho'(h).
\]
Then, it holds 
\[
g'(0)=\rho(h)\left(\rho'(h)-\sqrt{1+(\rho'(h))^2}\right)<0
\]
which contradicts the minimality of $\rho$. 
\\
\\
{\it Step 4.} Let us consider the system
\begin{equation}\label{problema}
\left\{\begin{array}{ll}
\displaystyle r=\Pi\cosh\frac{h}{\Pi},\\
\Pi>0.
\end{array}\right.
\end{equation}
Defining $\xi=\Pi/h$, the equation in \eqref{problema} can be rewritten as 
\begin{equation}
    \label{summaryfunction}
    \frac{h}{r}=\frac{1}{\xi \cosh \frac{1}{\xi}}=\left(\Phi\left(\frac{1}{\xi}\right)\right)^{-1}=:\mu(\xi).
\end{equation}
Notice that the problem \eqref{problema} has two different solutions $\Pi_0,\Pi_1$ with $\Pi_0>\Pi_1$ when $\frac{h}{r}<\frac{1}{m}$; it has one solution if $\frac{h}{r}=\frac{1}{m}$, which we still denote by $\Pi_0$; while it has no solutions if $\frac{h}{r}>\frac{1}{m}$. The function $\mu(\xi)$ is plotted in Figure \ref{summary}, together with the line $1/m$, and the two intersections $\Pi_0/h$ and $\Pi_1/h$ between $\mu$ and a generic line $h/r$, when $\frac{h}{r}<\frac{1}{m}$. When $\frac{h}{r}<\frac{1}{m}$, we aim to understand which solution between $\Pi_0$ or $\Pi_1$ of \eqref{problema} is the minimal one. Thus, we denote with $\rho_0,\rho_1 \colon [-h,h]\to \R$ the two catenaries given by 
\[
\rho_i(x)=\Pi_i\cosh\frac{x}{\Pi_i}, \quad i=0,1.
\]
When $\frac{h}{r}=\frac{1}{m}$, we still denote with $\rho_0$ the catenary associated to the unique solution $\Pi_0$. 

We claim that only one of these possibilities holds true: 
\begin{itemize}
\item[\rm(i)] $\rho=0$;
\item[\rm(ii)] $\rho=\rho_0$;
\item[\rm(iii)] $\rho=\rho_1$.
\end{itemize}
First of all, if $\rho(-h)=0$, hence also $\rho(h)=0$. Then by convexity, it must be $\rho=0$. If $\rho(-h)=\rho(h)>0$, then by Step 2, we have $\rho(-h)=\rho(h)=r$. This means that in a right neighborhood of $-h$ and in a left neighborhood of $h$ the minimizer $\rho$ has the form \eqref{forma} for some $\Pi>0$ and $x_0\in\R$. Since $\rho$ is even, it has to be $x_0=0$. Therefore, $\Pi$ solves problem \eqref{problema}, which means that either $\rho=\rho_0$ or $\rho=\rho_1$. 
\\
\\
{\it Step 5.} We claim that if $\frac{h}{r}<\frac{1}{m}$ then $\widetilde\E_0(\rho_0)<\widetilde\E_0(\rho_1)$. We explicitly have 
\[
\widetilde \E_0(\rho_i)=\Pi_ih+\Pi_i^2\sinh \frac{h}{\Pi_i}\cosh\frac{h}{\Pi_i}\stackrel{\eqref{problema}}{ = }\Pi_ih+ r\sqrt{r^2-\Pi_i^2}, \qquad i=0,1.
\]
Set $\lambda=h/r$ and $\xi_0=\Pi_0/h$, and $\xi_1=\Pi_1/h$. Then the inequality 
\[
\Pi_0h+ r\sqrt{r^2-\Pi_0^2}<\Pi_1h+ r\sqrt{r^2-\Pi_1^2}
\]
reads as 
\begin{equation}\label{confronto}
\xi_0\lambda^2+\sqrt{1-\xi_0^2\lambda^2}<\xi_1\lambda^2+\sqrt{1-\xi_1^2\lambda^2}.
\end{equation}
Since for $i = 1,2$
\[
\xi_i\lambda=\frac{\Pi_i}{r}=\left(\cosh\frac{h}{\Pi_i}\right)^{-1}=\left(\cosh\frac{1}{\xi_i}\right)^{-1},
\]
\eqref{confronto} easily reduces to 
\[
\frac{\text{sech}^2(1/\xi_0)}{\xi_0}+ \tanh(1/\xi_0)<\frac{\text{sech}^2(1/\xi_1)}{\xi_1}+ \tanh(1/\xi_1).
\]
Let $u \colon (0,+\infty) \to \R$ be the function given by 
$$
u(s):= \frac{\text{sech}^2(1/s)}{s}+ \tanh(1/s).
$$ 
Notice that $\Xi$, given by \eqref{eq:def_omega}, is the unique positive solution of the equation $u(s) = 1$. It is not difficult to see that $u(s) \geq 1$ for $s \leq \Xi$, while $u(s)<1$ if $s >\Xi$ (see, for instance, Figure \ref{functionu}).
\begin{figure}[htbp]
	\centering
	\includegraphics[width=0.65\textwidth]{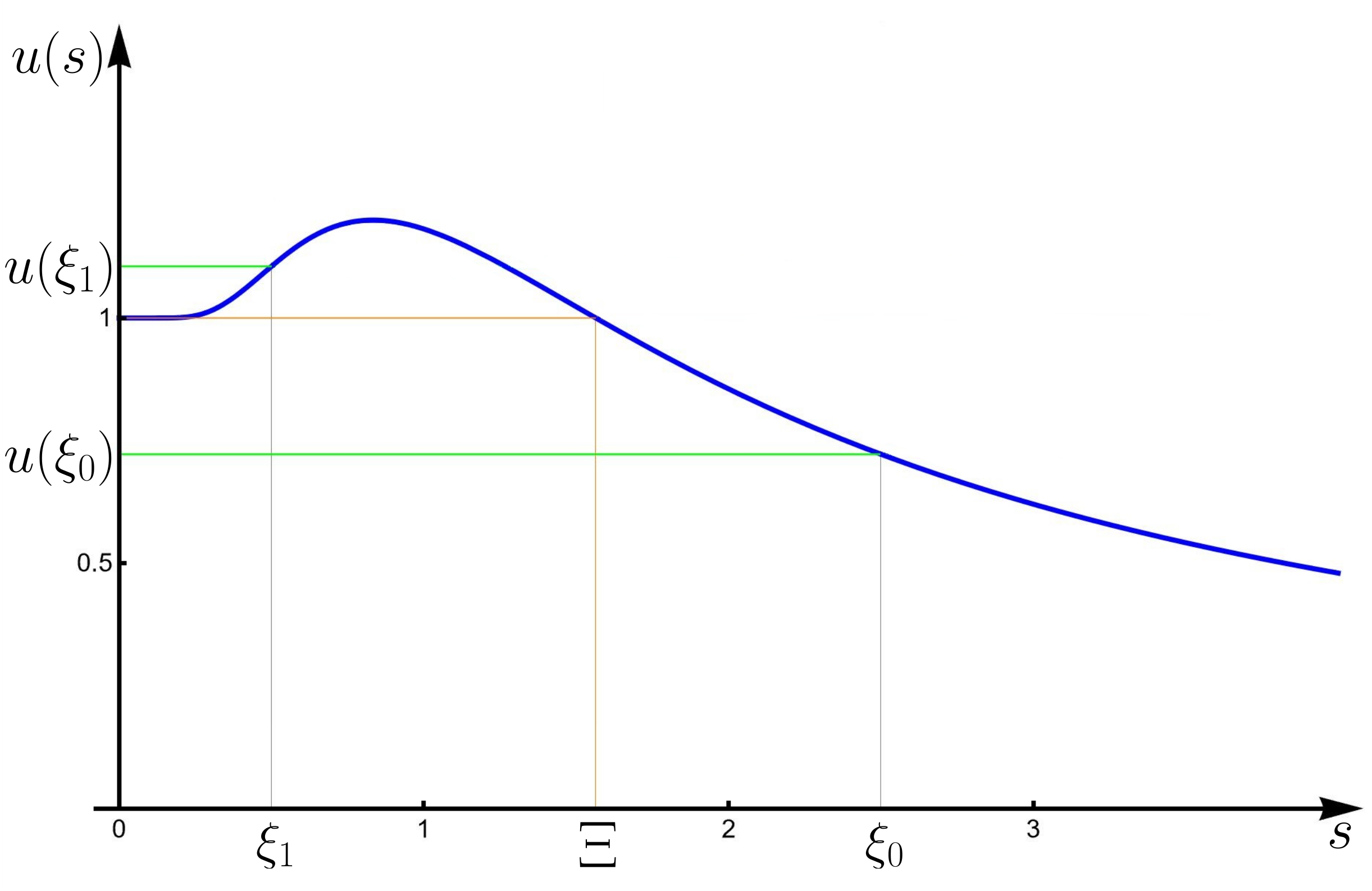}
	\caption{In blue, the function $u(s)$  and in orange the unique abscissa $\Xi$ with $u(\Xi)=1$. Grey lines refer to the two abscissas $\xi_0$ and $\xi_1$ corresponding to the two catenaries. Green lines correspond to $u(\xi_0)$ and $u(\xi_1)$.}
	\label{functionu}
\end{figure} 
Since $\Xi$ is the smallest value of possible $\xi_0$ such that $\frac{h}{r}\le\omega$, we have $\xi_1<\Xi<\xi_0$ which means that $u(\xi_0) < u(\xi_1)$ and this concludes the proof since it immediately implies that $\widetilde\E_0(\rho_0)<\widetilde\E_0(\rho_1)$. 
\\
\\
{\it Step 6.} We claim that: 
\begin{itemize}
    \item[\rm(i)] if $\frac{h}{r}<\omega$, then $\widetilde\E_0(\rho_0)<\widetilde \E_0(0)$; 
    \item[\rm(ii)] if $\frac{h}{r}=\omega$, then $\widetilde\E_0(\rho_0)=\widetilde\E_0(0)$; 
    \item[\rm(iii)] if $\frac{h}{r}>\omega$, then $\widetilde\E_0(\rho_0)>\widetilde \E_0(0)$.   
\end{itemize}
To conclude, it is sufficient to study the inequality $\Pi_0h+ r\sqrt{r^2-\Pi_0^2}<r^2$, since 
\[
\widetilde \E_0(\rho_0)=\Pi_0h+ r\sqrt{r^2-\Pi_0^2}, \qquad \widetilde \E_0(0)=r^2.
\]
Performing the same change of variable as above (see Step 5), we get
\[
\xi \tanh \frac{1}{\xi}+\text{sech}^2\frac{1}{\xi}<\xi,
\]
where $\xi = \Pi_0/h$ and it holds true if and only if $\xi>\Xi$. The condition $\xi>\Xi$ means that 
\[
\frac{h}{r}<\frac{1}{\Xi\cosh(1/\Xi)}=\omega.
\]
This completes the proof of (i). Just by inverting the inequality or studying the equality, it is immediate to show (ii) and (iii).
\end{proof}

\begin{remark}
\label{rem:analisi_catenaria}
A more careful analysis of the functional $\widetilde\E_0$ in terms of minimal surfaces of revolution gives the following well known conclusions \cite{Isenberg1978TheSO}, see also Figure \ref{summary}.

\begin{itemize}
\item[(a)] If $\frac{h}{r}<\omega$, then the catenoid generated by $\rho_0$ is the unique area minimizing surface spanning the two coaxial rings of radius $r$. 
\item[(b)] If $\frac{h}{r}=\omega$, then $\rho_0$ gives the catenoid as area minimizing surface, but the two discs of radius $r$ have the same area. The solution given by the two discs is known as Goldschmidt solution. 
\item[(c)] If $\frac{h}{r} \in(\omega,1/m]$, then the catenoid generated by $\rho_0$ still exists but it is only a local area minimizing surface, whereas the unique global area minimizing surface is the Goldschmidt solution. 
\item[(d)] If $\frac{h}{r}>1/m$, then no catenaries joining the boundary points exist, and the Goldschmidt solution is the unique global area minimizing surface.
\end{itemize}
Furthermore, it is found numerically that $\omega\simeq 0.528$ and $1/m\simeq 0.663$.

\end{remark}
\begin{figure}[htbp]
	\centering
	\includegraphics[width=0.6\textwidth]{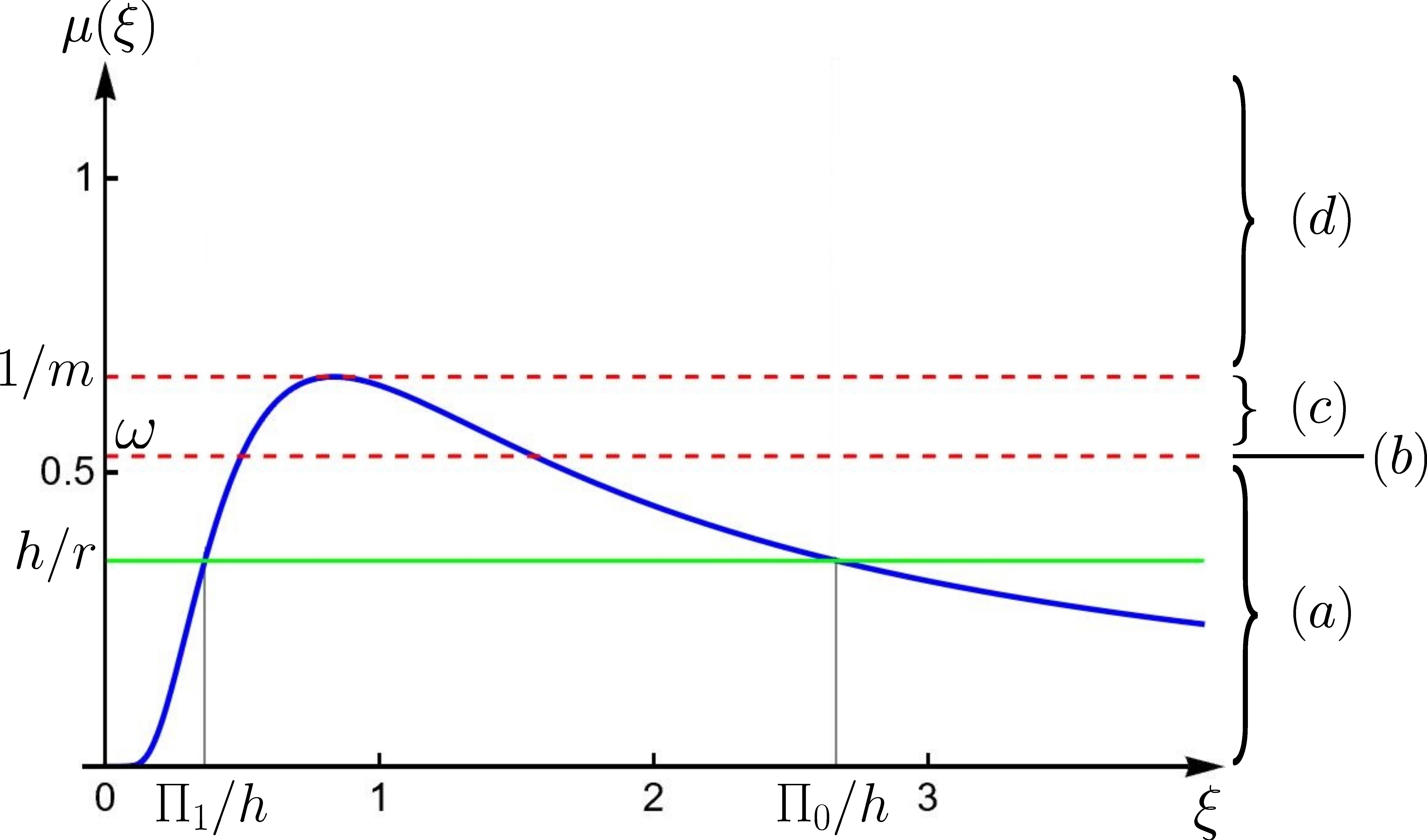}
	\caption{In blue, the function $\mu(\xi)$ defined in \eqref{summaryfunction} and in green a generic line at height $h/r$. Red lines refer to the two constants $1/m$ and $\omega$ introduced in Remark \ref{rem:analisi_catenaria} to distinguish the four possible cases. In the case $\frac{h}{r}<\frac{1}{m}$, the two intersection points $(\Pi_0/h, \mu(\Pi_0/h))$ and $(\Pi_1/h,\mu(\Pi_1/h))$, with $\Pi_0>\Pi_1$ are shown.}
	\label{summary}
\end{figure}
\section{Existence of a minimizer of \texorpdfstring{$\E_c$}{Ec}}
\label{sec:esistenza_minimi_Ec}

In this section we prove the existence of a minimizer of $\E_c$. A crucial step is the proof of compactness of minimizing sequences in $X$. It turns out that we are able to get compactness in $W^{1,\infty}(-h,h)$. Moreover, the functional $\E_c$ is continuous with respect to the strong $W^{1,1}$ convergence.

In what follows we will often use the function
\begin{equation}
    \label{eq:f}
    f:= \left\{
\begin{aligned}
    [0, +\infty) &\to \R\\
    x &\mapsto \frac{x^2}{\sqrt{1+x^2}}.
\end{aligned}
\right.
\end{equation}
Let $z_0:= \sqrt{\frac{1}{2} \left(\sqrt{5}-1\right)}$. Precisely, $z_0$ is the unique positive real such that $f'(z_0)=1$.

\begin{lemma}
For all $y \in [0, z_0]$ and for all $x \geq 0$, we have 
\begin{equation}\label{eq:dis_funzione_f}
f(x) - f(y) \geq f'(y)(x-y).
\end{equation}
\end{lemma}

\begin{proof}
The proof is just a computation: \eqref{eq:dis_funzione_f} reduces to
$$
\frac{1}{x-y}\left(\frac{x^2}{\sqrt{x^2+1}}-\frac{y^2}{\sqrt{y^2+1}}\right)\geq \frac{y
   \left(y^2+2\right)}{\left(y^2+1\right)^{3/2}},
$$
which is indeed satisfied for all $x\geq 0$ and $y \in [0, z_0]$.
\end{proof} 

\begin{lemma}\label{lemma_sup}
Assume \eqref{cond_catenaria}. Let $\Pi=\Pi(h,r)>0$ be the largest positive solution of the equation
\[
\Pi\cosh\frac{h}{\Pi}=r.
\]
Then 
\begin{equation}\label{sup}
\sup\left\{\sinh\frac{h}{\Pi}: \frac{h}{r}\le \omega\right\}<z_0.
\end{equation}
\end{lemma}

\begin{proof}
First of all, the largest value of $\sinh\frac{h}{\Pi}$ is the slope at $x=h$ of the catenary corresponding to the extremal condition $\frac{h}{r}=\omega$. More rigorously, it is easy to see that 
\[
\sup\left\{\sinh\frac{h}{\Pi}: \frac{h}{r}\le \omega\right\}=\sinh\frac{h}{\Pi_m}
\]
where 
\[
\Pi_m \cosh\frac{h}{\Pi_m}=\frac{h}{\omega}.
\]
Thus, we obtain 
$$
\sinh\frac{h}{\Pi_m}=\sqrt{\cosh^2\frac{h}{\Pi_m}-1}= \sqrt{\frac{h^2}{\Pi_m^2 \omega^2}-1}.
$$
The inequality 
\[
\sqrt{\frac{h^2}{\Pi_m^2 \omega^2}-1}<z_0
\]
is equivalent to
\begin{equation}\label{stima_Phi}
\frac{1}{\omega}>\Phi\left(\omega\sqrt{1+z_0^2}\right)
\end{equation}
where $\Phi$ is given by \eqref{Phi}. Using \eqref{stima_omega}, since $\frac{11}{5}<\sqrt 5<3$, we get
\[
\frac{52}{25\sqrt{10}}<\omega\sqrt{1+z_0^2}<\frac{53}{100}\sqrt 2<\text{argmin}\,\Phi.
\] 
Hence 
\[
\Phi\left(\omega\sqrt{1+z_0^2}\right)< \Phi\left(\frac{52}{25\sqrt{10}}\right)=\frac{25\sqrt{10}}{52}\cosh\frac{52}{25\sqrt{10}}<\frac{100}{52}<\frac{1}{\omega}
\]
where we have used the right inequality of \eqref{stima_omega}. This proves \eqref{stima_Phi}.
\end{proof}

\begin{lemma}
We have 
\begin{equation}\label{keyineq}
\frac{d}{dx}\left(\frac{1}{\rho_0}f'(\rho_0')\right)\ge 0.
\end{equation}
\end{lemma}

\begin{proof}
A straightforward computation shows that 
\[
\frac{d}{dx}\left(\frac{1}{\rho_0}f'(\rho_0')\right)=\frac{d}{dx}\left(v\left(\frac{\cdot}{\Pi_0}\right)\right)
\]
where 
$$
v:=\left\{
\begin{aligned}
    [0, +\infty)&\to \R\\
    x&\mapsto \frac{\sinh x(1+\cosh^2x)}{\cosh^4x}.
\end{aligned}
\right.
$$
We get
\[
v'(x) \geq 0\qquad \hbox{ if } \qquad x \in [0,\beta]
\]
where $\beta>0$ is such that $\cosh \beta=\sqrt{\frac{\sqrt{17}-1}{2}}$. Since 
\[
\frac{x}{\Pi_0}\le \frac{h}{\Pi_0}
\]
in order to have \eqref{keyineq}, it is sufficient to show that
\begin{equation}\label{stima_beta}
\frac{h}{\Pi_0} \le \beta.
\end{equation}
Precisely, \eqref{stima_beta} is equivalent to 
\[
\Phi\left(\frac{h}{\Pi_0}\right)\ge \Phi(\beta)
\]
which can be checked directly.
\end{proof}

\begin{proposition}\label{lemma_tilli}
Assume \eqref{cond_catenaria}. Let $\rho \in X$ be convex. For any $c>0$, it holds $\E_c(\rho\vee \rho_0)\le \E_c(\rho)$.
\end{proposition}

\begin{proof}
We deal only with a special case, the other ones can be treated in a similar way. Assume that there exists $a\in (-h,h)$ such that $\rho<\rho_0$ on $(a,h)$ and $\rho(a)=\rho_0(a)$, a graphical representation is given in Figure \ref{fig:caso_tilli_lemma}.
\begin{figure}[htbp]
		\centering
		\begin{tikzpicture}[> = latex, scale=1]
\begin{scope}[xscale=1.5]      
\begin{scope}
\path [clip] (-1.9,1.5) rectangle (1.9,-1.1);  
\draw [red, thick] (0,3) circle [x radius = 1.87cm, y radius = 4cm];
\end{scope}

\begin{scope}
\path [clip] (-2.2,1.5) rectangle (2.2,-1.6);  
\draw [blue, thick] plot [smooth] coordinates { (-1.73,1.5) (-1.65,1.2) (-1.6,1.03) (-1.5,0.75) (-1.25,0.15) (-1,-0.35) (-0.75,-0.75) (-0.5,-1.055) (-0.25,-1.195)  (0,-1.235) (0.25,-1.195) (0.5,-1.055) (0.75,-0.75) (1,-0.35) (1.25,0.15) (1.5,0.75) (1.6,1.03) (1.65,1.2) (1.73,1.5)};
\end{scope}
\filldraw [black] (0.94,-0.45) circle (1.2pt);
\filldraw [black] (-0.94,-0.45) circle (1.2pt);
\filldraw [black] (1.73,1.5) circle (1.2pt);
\draw [xshift=-.5em,dashed] (-1.56,-1.87) -- (-1.56,1.5);
\draw [xshift=-.5em,dashed] (1.92,-1.87) -- (1.92,1.5);
\path (-1.58,-1.9)node[below=1ex]{$x=-h$} 
(1.58,-1.9)node[below=1ex]{$x=h$};            
\draw node at (0,-0.7) {\textcolor{red}{$\rho_c$}};
\draw node at (0,-1.6) {\textcolor{blue}{$\rho_0$}};
\draw node at (1.3,-0.7) {\footnotesize{\textcolor{black}{$(a, \rho(a))$}}};
\draw node at (1.39,1.52) {\footnotesize{\textcolor{black}{$(h,r)$}}};
\end{scope}
\end{tikzpicture}
	\caption{Graphical representation of a possible case considered in Proposition \ref{lemma_tilli}.}
 \label{fig:caso_tilli_lemma}
\end{figure}
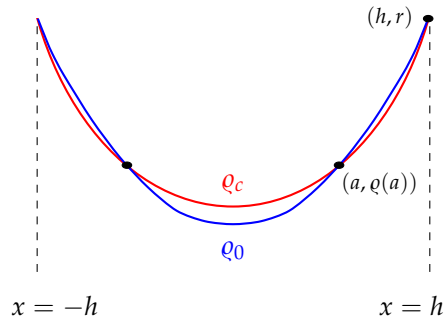
We show that 
\begin{equation}\label{tilli}
\bigintsss_a^h\left(\rho_0\sqrt{1+(\rho_0')^2}+c\frac{(\rho_0')^2}{\rho_0\sqrt{1+(\rho_0')^2}}\right)\,dx \le \bigintsss_a^h\left(\rho\sqrt{1+(\rho')^2}+c\frac{(\rho')^2}{\rho\sqrt{1+(\rho')^2}}\right)\,dx.
\end{equation}

First of all, if $\widetilde \rho \in X$ is given by 
\[
\widetilde \rho:=
\left\{
\begin{aligned}
&\rho_0 &&\hbox{on } [-h,a],\\
&\rho &&\hbox {on } (a,h],
\end{aligned}
\right.
\]
then, applying Theorem \ref{mainApp}, we get $\E_0(\rho_0)<\E_0(\widetilde \rho)$ which reduces \eqref{tilli} to
\[
\int_a^h\rho_0\sqrt{1+(\rho_0')^2}\,dx < \int_a^h\rho\sqrt{1+(\rho')^2}\,dx.
\]
As a consequence, since $\rho_0\ge \rho_c$ on $[a,h]$, in order to prove \eqref{tilli} it is sufficient to show that 
\begin{equation}\label{tilli1}
\int_a^h\frac{(\rho_0')^2}{\rho_0\sqrt{1+(\rho_0')^2}}\,dx \le \int_a^h\frac{(\rho')^2}{\rho_0\sqrt{1+(\rho')^2}}\,dx.
\end{equation}
Inequality \eqref{tilli1} reads as 
\begin{equation}
    \label{eq:tesi_to_show}
    \int_a^h\frac{1}{\rho_0}\left(f(\rho') - f(\rho_0')\right)\,dx \geq 0,
\end{equation}
where the function $f$ is defined in \eqref{eq:f}. Using \eqref{sup}, \eqref{keyineq} and \eqref{eq:dis_funzione_f}, we get 
\[
    \begin{aligned}
        \int_a^h\frac{1}{\rho_0}\left(f(\rho') - f(\rho_0')\right)\,dx &\geq \int_a^h\frac{1}{\rho_0}f'(\rho_0') \left(\rho' -\rho'_0\right)\,dx \\
        &= \frac{1}{\rho_0}f'(\rho_0')(\rho-\rho_0)\bigg|_a^h-\int_a^h \frac{d}{dx}\left(\frac{1}{\rho_0}f'(\rho_0')\right)(\rho-\rho_0)\,dx\\
        & = -\int_a^h\frac{d}{dx} \left(\frac{1}{\rho_0}f'(\rho_0')\right)(\rho-\rho_0)\,dx\ge 0
    \end{aligned}
\]
and this ends the proof of \eqref{tilli}.
\end{proof}

Thus, we are in position to show the main result of this section.

\begin{theorem}
    \label{thm:esistenza_minimo_Ec}
Assume \eqref{cond_catenaria}. For all $c >0$, the functional $\E_c$ admits a minimizer. 
\end{theorem}

\begin{proof}
Let $(\rho_j)$ be a  minimizing sequence for $\E_c$. By Proposition  \ref{lemmaconvexity-affine_con}, we can assume that $\rho_j$ is convex. In particular, $\rho_j\le r$.  
By Proposition \ref{lemma_tilli}, we can also assume that $\rho_j\ge \rho_0$ for any $j\in \N$, hence $\|\rho_j\|_\infty$ is bounded. For the same reason, by convexity we also have that $\|\rho_j'\|_\infty$ is bounded. Thus, we can apply Ascoli-Arzel\`a Theorem on $[-h,h]$ getting $\rho_j \to \overline\rho$ (up to a subsequence, not relabeled) uniformly on $[-h,h]$ for some $\overline\rho\in W^{1,1}(-h,h)$. 
In particular,  $\overline\rho$ is convex, and $\rho_0\le \overline\rho\le r$, and then $\overline\rho(\pm h) = r$. Hence, $\overline\rho\in X$.
Finally, thanks to Lemma \ref{lemma_conv} and Lemma \ref{lemma_cont}, the functional $\E_c$ is continuous with respect to the strong $W^{1,1}$ convergence. By the Direct method of the Calculus of Variations, we conclude that $\overline\rho$ is a minimizer of $\E_c$. 
\end{proof}
In what follows, for simplicity, for all $c\geq 0$, we denote any minimizer of $\mathcal{E}_c$ with $\rho_c$.

\begin{remark}\label{bound_derivata}
Let $\rho_c$ be a minimizer of $\E_c$. Combining Lemma \ref{lemma_sup} with Proposition \ref{lemma_tilli}, we can say that $|\rho_c'|<z_0$.
\end{remark}

\section{Properties of minimizers of \texorpdfstring{$\E_c$}{Ec}}
\label{sec:proprieta_geometriche}

In this section, we conclude the proof of Theorem \ref{main_total}. 

\begin{theorem}
Let $\rho_c$ be a minimizer of $\E_c$. Then $\rho_c\in C^2([-h,h])$ and \eqref{EL} and \eqref{eq:integrale_primo} hold true. Moreover, $\rho_c$ is even and convex.
\end{theorem}

\begin{proof}
For simplicity of notation, we let $\rho=\rho_c$. Let $\phi\in C^1_{c}(-h,h)$. For any $\sigma >0$ small enough, we have that $\rho+t\phi \in X$ for each $t \in (- \sigma, \sigma)$. Then, we easily obtain 
\[
0=\frac{d}{dt}_{\big|_{t=0}} \E_c(\rho + t\phi) =\bigintsss_{-h}^h \left(\frac{\rho \rho'}{\sqrt{1 + (\rho')^2}} + c\, \frac{2 \rho \rho'(1 + (\rho')^2)  -\rho (\rho')^3}{\rho^2 \left(1 + (\rho')^2\right)^{\frac{3}{2}}}-\Theta\right)\phi'\, dx
\]
where
    $$\Theta(x):= \bigintsss_{-h}^x \sqrt{1 + (\rho')^2} -c\, \frac{(\rho')^2}{\rho^2\sqrt{ 1 + (\rho')^2}}\, d\vartheta.
    $$
    Applying Du Bois-Reymond's Lemma, we deduce that 
\begin{equation}
    \label{eq:integrale_primo_variazionale}
    \frac{\rho \rho'}{\sqrt{1 + (\rho')^2}} + c\,\frac{2 \rho \rho'(1 + (\rho')^2)  -\rho (\rho')^3}{\rho^2 (1 + (\rho')^2)^{\frac{3}{2}}}- \Theta = \Gamma, \qquad \text{a.e.\,on $(-h,h)$},
\end{equation}
where $\Gamma\in \R$. 
Using the Implicit Function Theorem, we obtain 
\[
\rho'=g(\rho,\Theta+\Gamma), \qquad \text{a.e.\,on $(-h,h)$},
\]
for some smooth function $g$. As a consequence, $\rho'\in W^{1,1}(-h,h)$, which means $\rho\in W^{2,1}(-h,h)$. Differentiating \eqref{eq:integrale_primo_variazionale}, we get  
$$
(1+(\rho')^2)((c+\rho^2) (\rho')^2+\rho^2)=\rho\rho'' (\rho^2(\rho')^2+c(2-(\rho')^2)+\rho^2), \qquad \text{a.e.\,on $(-h,h)$}.
$$
Notice that the left-hand side is strictly positive, while the term
\[
\rho^2(\rho')^2+c(2-(\rho')^2)+\rho^2
\]
is strictly positive as well since $|\rho'|<z_0<\sqrt 2$. Hence 
\[
\rho''=\frac{(1+(\rho')^2)((c+\rho^2) (\rho')^2+\rho^2)}{\rho(\rho^2(\rho')^2+c(2-(\rho')^2)+\rho^2)}> 0
\]
which means that $\rho \in C^2([-h,h])$ and that $\rho$ is strictly convex. Since the Lagrangian is autonomous, as a standard consequence the Euler-Lagrange equation \eqref{EL} has the prime integral 
\[
c\frac{(\rho')^2}{\rho(1+(\rho')^2)^{3/2}}-\frac{\rho}{\sqrt{1+(\rho')^2}}=-\Pi_c
\]
for some constant $\Pi_c>0$. In particular, 
\[
c\frac{(\rho'(\pm h))^2}{r(1+(\rho'(\pm h))^2)^{3/2}}-\frac{r}{\sqrt{1+(\rho'(\pm h))^2}}=-\Pi_c
\]
from which $\rho'(-h)=-\rho'(h)$. By applying again the Implicit Function Theorem, we also deduce that $\rho'(h)=\tau(r,c)$ for some smooth function $\tau$. As a consequence, $\rho$ is a solution of the backward second order Cauchy problem
\[
\left\{\begin{array}{ll}
\displaystyle \rho''=\frac{(1+(\rho')^2)((c+\rho^2) (\rho')^2+\rho^2)}{\rho(\rho^2(\rho')^2+c(2-(\rho')^2)+\rho^2)}\\
\rho(h)=r\\
\rho'(h)=\tau(r,c).
\end{array}\right.
\]
On the other hand, $\tilde\rho(x)=\rho(-x)$ solves the same problem: notice that $\Pi_c$ is the minimum of $\tilde \rho$ as well, which means that $\tilde\rho'(h)=-\rho'(-h)=\rho'(h)=\tau(r,c)$. Hence $\rho=\tilde \rho$, namely $\rho$ is even. In particular, $\Pi_c=\rho_c(0)$ and the proof is complete.
\end{proof}

\begin{theorem}
Let $\rho_c$ be a minimizer of $\E_c$. Then $\rho_c > \rho_0$ on $(-h,h)$. 
\end{theorem}
\begin{proof}
By Lemma \ref{lemma_tilli}, we already know that $\rho_c \ge \rho_0$ on $(-h,h)$. Assume, by contradiction, that there is $\overline x\in (-h,h)$ such that   $\rho_c(\overline x)=\rho_0(\overline x)$. 
    \\
    \\
    {\it Step 1.} We claim that $\overline x = 0$. First of all, since $\rho_c(\overline x)=\rho_0(\overline x)$, since both $\rho_c$ and $\rho_0$ are of class $C^1$ and since $\rho_c\ge \rho_0$ we deduce that $\rho_c'(\overline x)=\rho_0'(\overline x)$. Taking into account that 
    \[
    \rho_c(\overline x)=\rho_0(\overline x)=\Pi_0\cosh\frac{\overline x}{\Pi_0}, \qquad \rho_c'(\overline x)=\rho_0'(\overline x)=\sinh\frac{\overline x}{\Pi_0}
    \]
    and substituting in \eqref{eq:integrale_primo} we obtain 
\[
     c\frac{\sinh^2\frac{\overline x}{\Pi_0}}{\Pi_0\cosh\frac{\overline x}{\Pi_0}\left(1+\sinh^2\frac{\overline x}{\Pi_0}\right)^{3/2}}-\frac{\Pi_0\cosh\frac{\overline x}{\Pi_0}}{\sqrt{1+\sinh^2\frac{\overline x}{\Pi_0}}}=-\rho_c(0)
     \]
     which simplifies in 
        \[
         c\frac{\sinh^2\frac{\overline x}{\Pi_0}}{\Pi_0\cosh^4\frac{\overline x}{\Pi_0}}=-\rho_c(0) + \Pi_0\le 0.
    \]
This means that $\sinh^2\frac{\overline x}{\Pi_0}=0$ namely $\overline x=0$.
    \\
    \\
    {\it Step 2.} We now conclude the proof proving that assuming $\Pi_0=\rho_c(0)$, we get a contradiction. Since $\rho'_c(0) =0 =  \rho'_0(0)$ and since $\rho_c\ge \rho_0$, by Taylor expansion we obtain 
    \begin{equation}\label{taylor}
    \rho''_c(0) \geq \rho_0''(0)=\frac{1}{\Pi_0}.
    \end{equation}
    Using \eqref{EL}, we easily find 
    \[
    \rho_c''(0)=\frac{\rho_c(0)}{2 c + \rho_c(0)^2}.
    \]
    As a consequence, \eqref{taylor} reduces to
    \begin{equation}
\label{eq:dis_derivata_seconda}
\frac{\rho_c(0)}{2 c + \rho_c(0)^2}\geq \frac{1}{\Pi_0}
    \end{equation}
    that implies $\rho_c(0)<\Pi_0$, which is a contradiction.
\end{proof}

Finally, we show that $\rho_c$ uniformly converges to $r$ as $c$ goes to infinity.

\begin{proposition}
    \label{prop:gamma_convergence}
    Let $(c_j)$ be a positive sequence with $c_j \to +\infty$. Let $\rho_j$ be a minimizer of $\E_{c_j}$. Then $\rho_j \to r$ uniformly on $[-h,h]$  as $j \to +\infty$.
  
\end{proposition}
\begin{proof}
Define 
$$
Y:= \{\rho \in W^{1,1}(-h,h), \rho \hbox{ convex }, \rho_0 \leq \rho\leq r, \rho(-h) =\rho(h)=r\}
$$ 
and the energy functionals $\mathcal F_j,\mathcal F_\infty\colon Y\to \R$ given by 
\[
\mathcal F_j(\rho)=\frac{1}{c_j}\mathcal{E}_{c_j} (\rho), \qquad \mathcal F_\infty(\rho)=\bigintsss_{-h}^h \frac{(\rho')^2}{\rho \sqrt{1 + (\rho')^2}}\, dx.
\]
Obviously, $\rho=r$ is the unique minimizer of $\mathcal{F}_{\infty}$. 
By Ascoli-Arzel\`a Theorem, $Y$ is compact with respect to the uniform convergence, thus $\rho_j \to \rho_{\infty}$ uniformly on $[-h,h]$ as $j \to \infty$ for some $\rho_\infty\in Y$. On the other hand, combining Lemma \ref{lemma_conv} and Lemma \ref{lemma_cont}, we obtain
    \begin{equation}
    \label{eq:liminf}
        \liminf_{j\to+\infty} \mathcal{F}_{j}(\rho_j) \geq  \liminf_{j\to+\infty} \int_{-h}^{h} \frac{(\rho_j')^2}{\rho_j \sqrt{1 + (\rho')^2}} = \int_{-h}^{h} \frac{(\rho_\infty')^2}{\rho_\infty \sqrt{1 + (\rho_\infty')^2}}=\mathcal F_\infty(\rho_\infty).
    \end{equation}
Let $\rho \in Y$. Then, since $\rho_j$ is also a minimizer of $\mathcal F_j$, using \eqref{eq:liminf}, we deduce that 
$$
\mathcal F_\infty(\rho_\infty) \leq \liminf_{j\to+\infty} \mathcal F_j(\rho_j) \leq \liminf_{j\to+\infty} \frac{1}{c_j} \E_{c_j} (\rho) = \mathcal F_\infty(\rho).
$$
We get that $\rho_{\infty}=r$, and this concludes the proof.
\end{proof}

\begin{appendices}
    \section{Derivation of the energy functional \texorpdfstring{$\E_c$}{Ec}}
    \label{sec-phy}
    
In this Appendix, we derive the energy functional \eqref{energy_functional} specializing 
\begin{equation}\label{funz_app}
    \mathcal{E}(\n, S)=\int_S \left(\gamma+\frac{\kappa}{2}|\tens D \n|^2\right)dA
\end{equation}
in the setting of a revolution surface $S$ spanning two coaxial rings of radius $r$ placed at distance $2h$. We remember that here $\tens D \n$ stands for the covariant derivative of $\n$.
First of all, we parametrize $S$ by means of $\vect Z \colon [0,2\pi] \times [-h,h] \to \R$ given by 
\[
\vect{Z}(\varphi,x)=(\rho(x) \cos \varphi,\rho(x)\sin \varphi,x),
\]
where $\rho:[-h,h]\to (0,+\infty)$ is a smooth function with $\rho(\pm h)=r$. In this setting, we have 
\[
\vect{Z}_\varphi=(-\rho\sin \varphi, \rho \cos\varphi,0), \qquad \vect{Z}_x=(\rho'\cos \varphi, \rho' \sin\varphi,1).
\]
Then, 
\[
dA=|\vect Z_\varphi \times \vect Z_x|d\varphi dx=\rho\sqrt{1+(\rho')^2}\,d\varphi dx.
\]
Since  
\[
\vect{e}_1=\frac{\vect{Z}_\varphi}{|\vect{Z}_\varphi|}, \qquad 
\vect{e}_2=\frac{\vect{Z}_x}{|\vect{Z}_x|}
\]
form an orthonormal basis on the tangent plane to $S$ and since $\n$ is a unit vector constrained to be tangent to the surface, we can rewrite $\n$ as $\n=\cos \alpha \vect{e}_1+\sin \alpha \vect{e}_2$, where $\alpha=\alpha(x,\varphi) \in [0, \pi]$ is the positive oriented angle formed by the vector $\n$ with $\vect e_1$. Since $|\vect Z_\varphi|=\rho$ and $|\vect Z_x|=\sqrt{1+(\rho')^2}$, we explicitly have 
$$
\n=\frac{\cos \alpha}{\rho}\vect{Z}_\varphi+\frac{\sin \alpha}{\sqrt{1+(\rho')^2}}\vect{Z}_x.
$$
We need to compute the covariant derivative of $\n$, so that from now on we will assume $\alpha$ smooth. Using the notation $x_1=\varphi$ and $x_2=x$, the metric tensor is given by  
\[
(g_{ij})=\begin{pmatrix}
\rho^2 & 0\\
0 & 1+(\rho')^2
\end{pmatrix}
\]
with inverse 
\[
(g^{ij})=\begin{pmatrix}
1/\rho^2 & 0\\
0 & 1/(1+(\rho')^2)
\end{pmatrix}.
\]
The Christoffel symbols of the Levi-Civita connection on $S$  are given by  
$$
\Gamma_{ij}^m=\frac{1}{2}\sum_{k=1}^2g^{km}\left(\frac{\partial g_{jk}}{\partial x_i}+\frac{\partial g_{ki}}{\partial x_j}-\frac{\partial g_{ij}}{\partial x_k}\right), \qquad i,j,m=1,2.
$$
The only non-zero Christoffel symbols are 
$$
\Gamma_{11}^2=-\frac{\rho\rho'}{1+(\rho')^2}, \qquad \Gamma_{12}^1=\Gamma_{21}^1=\frac{\rho'}{\rho}, \qquad \Gamma_{22}^2=\frac{\rho'\rho''}{1+(\rho')^2}.
$$
Now, if $\vect X=X^1\vect Z_\varphi+X^2 \vect Z_x$ the covariant derivative of $\vect X$ is the tensor given by 
\[
(\tens D \vect X)_j^k=\frac{\partial X^k}{\partial x_j}+\sum_{i=1}^2\Gamma_{ij}^kX^i, \qquad j,k=1,2.
\]
As a consequence we find
\[
\begin{aligned}
    (\tens D \n)_1^1&=-\frac{\alpha_\varphi\sin\alpha}{\rho}+\frac{\rho'\sin\alpha}{\rho\sqrt{1+(\rho')^2}}, \qquad &&(\tens D \n)_1^2=\frac{\alpha_\varphi\cos\alpha}{\sqrt{1+(\rho')^2}}-\frac{\rho'\cos\alpha}{1+(\rho')^2}, \\
(\tens D \n)_2^1&=-\frac{\alpha_x\sin\alpha}{\rho}, \qquad &&(\tens D \n)_2^2=\frac{\alpha_x\cos\alpha}{\sqrt{1+(\rho')^2}}.
\end{aligned}
\]
Finally, we get  
\[
|\tens D \n|^2=(\tens D \n)_\ell^j(\tens D \n)_m^hg_{jh}g^{\ell m}=\frac{\alpha_x^2}{1 + (\rho')^2} + \frac{\alpha_\varphi^2}{\rho^2}-\frac{2\rho'\alpha_\varphi}{\rho^2 \sqrt{1 + (\rho')^2}} +\frac{(\rho')^2}{\rho^2 (1 +(\rho')^2)}.
\]
Therefore, if we use Fubini's Theorem, the energy functional \eqref{funz_app} takes the form 
\[
\begin{aligned}
 \mathcal{E}(\n, S)&=\int_{-h}^h\int_0^{2\pi}\left(\gamma\rho\sqrt{1+(\rho')^2}+\frac{\kappa}{2}\left(\frac{\alpha_x^2}{1 + (\rho')^2} + \frac{\alpha_\varphi^2}{\rho^2}-\frac{2\rho'\alpha_\varphi}{\rho^2 \sqrt{1 + (\rho')^2}} +\frac{(\rho')^2}{\rho^2 (1 +(\rho')^2)}\right)\rho\sqrt{1+(\rho')^2}\right)\,d\varphi dx\\
 &=I_1+I_2+I_3+I_4
 \end{aligned}
 \]
 where 
\begin{equation}
    \label{eq:simpl_funzionale}
    \begin{aligned}
 I_1&=2\pi\gamma\int_{-h}^h\left(\rho\sqrt{1+(\rho')^2}+\frac{\kappa}{2\gamma}\frac{(\rho')^2}{\rho \sqrt{1 +(\rho')^2}}\right)\,dx, \qquad &&I_2=\frac{\kappa}{2}\int_{-h}^h\frac{\rho}{\sqrt{1 + (\rho')^2}}\int_0^{2\pi}\alpha_x^2\,d\varphi dx,\\
   I_3&=\frac{\kappa}{2}\int_{-h}^h\frac{\sqrt{1+(\rho')^2}}{\rho}\int_0^{2\pi}\alpha_\varphi^2\,d\varphi dx, \qquad &&I_4=-\frac{\kappa}{2}\int_{-h}^h\frac{2\rho'}{\rho}\int_0^{2\pi}\alpha_\varphi\,d\varphi dx. 
\end{aligned}
\end{equation}
First of all, $I_4=0$: $\n$ is smooth, hence $\alpha(x,0)=\alpha(x,2\pi)$ for any $x\in [-h,h]$. Since $\alpha$ is not subjected to any boundary condition, in order to minimize $\mathcal E$ it is therefore convenient to choose $\alpha_x=\alpha_\varphi=0$ and minimize $I_1$, which is, up to $2\pi\gamma$, the functional \eqref{energy_functional}, having defined $c := \frac{\kappa}{2 \gamma}$
\end{appendices}

\bigskip

\section*{Acknowledgements}
The authors thank Paolo Tilli for his fundamental advice in proving Proposition \ref{lemma_tilli}. Moreover, the authors thank Giuseppe Buttazzo and Marco Degiovanni for fruitful discussions. GB acknowledges the MIUR Excellence Department Project awarded to the Department of Mathematics, University of Pisa, CUP I57G22000700001. 
CL and AM acknowledge Gruppo Nazionale per la Fisica Matematica (GNFM) of Istituto Nazionale di Alta Matematica (INdAM).

\section*{Funding}
\begin{itemize}
\item GB is supported by European Research Council (ERC), under the European Union's Horizon 2020 research and innovation program, through the project ERC VAREG - {\em Variational approach to the regularity of the free boundaries} (grant agreement No. 853404).
\item GB and LL are supported by Gruppo Nazionale per l'Analisi Matematica, la Probabilit\`a e le loro Applicazioni (GNAMPA) of Istituto Nazionale di Alta Matematica (INdAM) through the INdAM-GNAMPA project 2024 CUP E53C23001670001.
\item LL is supported by European Union - Next Generation EU - Research Project Prin2022 PNRR of National Relevance P2022KHFNB granted by the Italian MUR.
\item CL is suppported by the MICS (Made in Italy – Circular and Sustainable) Extended Partnership and received funding from the European Union Next-Generation EU (PIANO NAZIONALE DI RIPRESA E RESILIENZA (PNRR) – MISSIONE 4 COMPONENTE 2, INVESTIMENTO 1.3 – D.D. 1551.11-10-2022, PE00000004). 
\end{itemize}

\section*{Conflict of interest statement}
The authors declare that they have no known competing financial interests or personal relationships that could have appeared to influence the work reported in this paper

\printbibliography
\end{document}